\documentclass[11pt]{article}
\usepackage{amsthm,amsmath,amssymb}
\usepackage{latexsym}
\usepackage{pst-node}
\usepackage{subfigure}
\usepackage{multirow}
\usepackage{tikz}
\usepackage{pgf}
\usepackage[colorlinks=true,citecolor=black,linkcolor=black,urlcolor=blue]{hyperref}
\usepackage[hmargin=2.8cm,vmargin=2.5cm]{geometry}

\theoremstyle{plain}
\newtheorem{theorem}{Theorem}
\newtheorem{lemma}[theorem]{Lemma}
\newtheorem{corollary}[theorem]{Corollary}
\newtheorem{proposition}[theorem]{Proposition}

\newtheorem{observation}[theorem]{Observation}
\newtheorem{claim}{Claim}

\theoremstyle{definition}
\newtheorem{definition}[theorem]{Definition}
\newtheorem{question}[theorem]{Question}

\newcommand{\M}{\gamma^{\text{\tiny{ID}}}}
\newcommand{\LD}{\gamma^{\text{\tiny{LD}}}}
\newcommand{\DS}{\gamma}

\tikzstyle{graphnode}=[draw,shape=circle,draw=black,minimum size=0.5pt,inner sep=1.5pt]
\tikzstyle{edgecode}=[line width =3 pt]
\tikzstyle{edge}=[thin]

\title{Locating-dominating sets and identifying codes\\in graphs of girth at least 5}
\author{Camino Balbuena\thanks{This research was supported by the Ministry of Education and Science, Spain, and the European Regional Development Fund (ERDF) under project  MTM2011-28800-C02-02, and under the Catalonian Government project 1298 SGR2009.}\\
\small Departament de Matem\`atica Aplicada III\\[-0.8ex]
\small Universitat Polit\`ecnica de Catalunya, Barcelona, Spain\\
\small\tt m.camino.balbuena@upc.edu\\
\and
Florent Foucaud\thanks{Part of this research was done while this author was a postdoctoral researcher at the Departament de Matem\`atica Aplicada III of Universitat Polit\`ecnica de Catalunya, Barcelona, Spain, at the Department of Mathematics of University of Johannesburg, South Africa, and at LAMSADE - PSL Universit\'e Paris-Dauphine, France.}~\footnotemark[1]\\
\small LIMOS - CNRS UMR 6158\\[-0.8ex]
\small Universit\'e Blaise Pascal, France\\
\small\tt florent.foucaud@gmail.com\\
\and
Adriana Hansberg\thanks{Research partially supported by the program UNAM-DGAPA-PAPIIT under project IA103915 and by \emph{Fondo Sectorial de Investigaci\'on para la Educaci\'on} under project 219775 CB-2013-01.}~\footnotemark[1]~\thanks{Part of this research was done while this author was a postdoctoral researcher at the Departament de Matem\`atica Aplicada III of Universitat Polit\`ecnica de Catalunya, Barcelona, Spain.}\\
\small Instituto de Matem\'aticas\\[-0.8ex]
\small Universidad Nacional Aut\'onoma de M\'exico, M\'exico\\
\small\tt ahansberg@im.unam.mx\\
}

\date{April 9, 2015}

\begin{document}
\maketitle

\begin{abstract}
Locating-dominating sets and identifying codes are two closely related notions in the area of separating systems. Roughly speaking, they consist in a dominating set of a graph such that every vertex is uniquely identified by its neighbourhood within the dominating set. In this paper, we study the size of a smallest locating-dominating set or identifying code for graphs of girth at least~$5$ and of given minimum degree. We use the technique of vertex-disjoint paths to provide upper bounds on the minimum size of such sets, and construct graphs who come close to meet these bounds.
%  \vspace{\baselineskip}
% {\bf Key words.}  Identifying codes; locating-dominating sets; dominating sets; path covers; girth; minimum degree.
% \vspace{\baselineskip}
\end{abstract}

\newcommand{\minitab}[2][l]{\begin{tabular}{#1}#2\end{tabular}}

\section{Introduction}

Various forms of distinguishing problems in graphs arising from several applications have been studied. Imagine a setting where one wants to detect a hazard in a network (graph) using simple local detectors. Every network node should be within reach of some detector, say at graph distance at most~$1$: in this case the detectors must form a dominating set. If, in addition, one wants to be able to precisely locate the hazard, every node must be uniquely determined by the set of detectors monitoring (dominating) it. This is the notion of a locating-dominating set or an identifying code (depending on whether the detector nodes should be distinguished themselves).

Since the introduction of locating-dominating sets by Slater~\cite{S87,S88} and identifying codes by Karpovsky, Chakrabarty and Levitin~\cite{KCL98}, these concepts have been widely studied and applied to hazard- or fault-detection in networks and facilities~\cite{KCL98,UTS04}, routing~\cite{LTCS07}, as well as in relation with graph isomorphism~\cite{B80} and logical characterizations of graphs~\cite{KPSV04}. An online bibliography on these topics is maintained by Lobstein~\cite{biblio}. We remark that these problems belong to the more general set of distinguishing or separating problems in graphs and hypergraphs; see the concept of hypergraph \emph{separating systems}~\cite{BS07,R61} (which is also known under the name of \emph{test covers}~\cite{DHHHLRS03,MS85} or \emph{discriminating codes}~\cite{CCCHL08}, and is related to a celebrated theorem of Bondy~\cite{B72}).

In this paper, we study locating-dominating sets and identifying codes in graphs of girth at least~$5$ (that is, containing no triangle or $4$-cycle). Their behaviour in this class is quite different from the class of graphs with girth~$3$ or~$4$. We are able to give upper bounds on the smallest size of such sets in terms of the order of the graph, and discuss the tightness of our bounds.

\medskip
\noindent\emph{Definitions.} All graphs in this paper will be undirected and finite. The order of a graph will be denoted by the letter $n$. The open and closed neighbourhoods of a vertex $x$ are denoted $N(x)$ and $N[x]$, respectively, and the \emph{degree} of $x$ is the size of its open neighbourhood. A graph is \emph{cubic} if all its vertices have degree~$3$. A path along vertices $x_1,\ldots,x_k$ is denoted $x_1-\ldots-x_k$. The \emph{order} of a path is the number of its vertices, and its \emph{length} is the number of its edges (that is, its order minus one). We may also denote the concatenation of two paths $P,P'$ by $P-P'$. A \emph{Hamiltonian path} of a graph is a path containing all its vertices. Given a set $X$ of vertices in a graph $G$, $G[X]$ denotes the subgraph of $G$ induced by $X$. A graph $G$ is \emph{vertex-transitive} if, given any two vertices $x$ and $y$, there is an automorphism of $G$ mapping $x$ to $y$.

In a graph $G$, a vertex \emph{dominates} itself and all its neighbours. A set $D$ of vertices dominates vertex $x$ if some vertex of $D$ dominates $x$. Similarly, $D$ \emph{2-dominates} $x$ if at least two distinct vertices of $D$ dominate $x$. Set $D$ is called a \emph{dominating set} if $D$ dominates all vertices in $V(G)$.  If a vertex $x$ belongs to the symmetric difference $N[u]\Delta N[v]$ (i.e. $x$ dominates exactly one of $u,v$) we say that $x$ \emph{separates} $u$ from $v$.

We have the following definitions of the core concepts of this paper:

\begin{definition}[\cite{KCL98,S87,S88}]
Given a graph $G$, a subset $C$ of vertices of $V(G)$ which is both a dominating set and such that all vertex-pairs in $V(G)\setminus C$ are separated by some vertex of $C$ is called a \emph{locating-dominating set} of $G$. If \emph{all} vertex-pairs in $V(G)$ are separated by some vertex of $C$, it is called an \emph{identifying code} of $G$.
\end{definition}

Note that a graph always has a locating-dominating set, but it has an identifying code (it is \emph{identifiable}) if and only if it has no \emph{twins}, i.e. vertices with the same closed neighbourhood. However, for triangle-free graphs (and thus for graphs of girth at least~$5$), twins cannot have any common neighbour, leading to the following observation:

\begin{observation}
A triangle-free graph is identifiable if and only if it has no connected component with two vertices.
\end{observation}

The minimum size of a dominating set, a locating-dominating set, and an identifying code of a graph $G$ are called the \emph{domination number} $\DS(G)$, the \emph{location-domination number} $\LD(G)$ and the \emph{identifying code number} $\M(G)$ of $G$, respectively. If $G$ is identifiable we have $\DS(G)\leq\LD(G)\leq\M(G)$.

\medskip
\noindent\emph{Related work.} A classic result in domination due to Ore~\cite{O62} is that for every graph $G$ of order $n$ with minimum degree at least~$1$, $\DS(G)\leq\frac{n}{2}$. Later, McCuaig and Shepherd~\cite{MS89} proved that besides seven exceptional graphs, if $G$ is connected and has minimum degree at least~$2$, then $\DS(G)\leq\frac{2n}{5}$. For minimum degree at least~$3$, Reed~\cite{R96} proved the bound $\DS(G)\leq\frac{3n}{8}$. More generally, it is known that any graph $G$ with minimum degree~$\delta$ has domination number $\DS(G) =  O\left(\frac{\log\delta}{\delta}\right)n$ (see~\cite{AS08}), and this bound is asymptotically tight \cite{A90}. On the other hand, for connected cubic graphs with $n\geq 9$, Kostochka and Stodolsky~\cite{KS09} proved that $\DS(G)\leq\frac{4n}{11}$.

A bound of this form does not exist for locating-dominating sets or identifying codes. Indeed, $d$-regular graphs with locating-dominating number and identifying code number of the form $n\left(1-\frac{1}{\Theta(d)}\right)$ were constructed by the second author and Perarnau~\cite{FP11}. However, these constructions contain either triangles or $4$-cycles, and the same authors showed that for any graph $G$ of order $n$, girth at least~$5$ and minimum degree~$\delta$, an (asymptotically tight) upper bound of the form $\LD(G)\leq\M(G) = O\left(\frac{\log\delta}{\delta}\right)n$ (similar to the one for dominating sets) holds. However, for small values of $\delta$ the bound of~\cite{FP11} is not meaningful; when $\delta=2$, the second author showed the bound $\M(G)\leq\frac{7n}{8}=0.875n$ in his PhD thesis~\cite{F}.

In this paper, we study the following question:

\begin{question}
What are tight upper bounds on $\LD(G)$ and $\M(G)$ for graphs $G$ of given (small) minimum degree $\delta\geq 2$ and girth at least~$5$?
\end{question}

\renewcommand{\arraystretch}{0.9}
\begin{table}[ht!]
\centering
\scalebox{0.85}{
    \begin{tabular}{c||c|c|c||c|c|c}
 &  \multicolumn{3}{c}{\textbf{Location-domination number}} &   \multicolumn{3}{c}{\textbf{Identifying code number}}\\
 & \multicolumn{3}{|c||}{} & \multicolumn{3}{|c}{}\\
 & \multicolumn{2}{c|}{largest known examples} &   & \multicolumn{2}{c|}{largest known examples} & \\[0.1cm]
& & & upper bound & & & upper bound\\
& small & arb. large & & small & arb. large & \\[0.2cm]
\hline
\hline
& & & & & & \\[-0.2cm]
\multirow{2}*{$\delta=2$}   & $0.5$ & $0.5-\epsilon$  & \multirow{5}*{\minitab[c]{$0.5$\\Thm.~\ref{thm:vdp-LD}}} & $\frac{5}{7}$ & $0.6-\epsilon$ & \multirow{5}*{\minitab[c]{$\frac{5}{7} < 0.715$\\Thm.~\ref{thm:vdp-ID-delta-2}}}\\
& $C_6$ \cite{S88} & Prop.~\ref{prop:example-LD-n/2} &  & $C_7$ \cite{BCHL04} & Prop.~\ref{prop:example-ID-3n/5} & \\
& & & & & & \\[-0.3cm]\cline{1-3}\cline{5-6}
& & & & & & \\[-0.2cm]
\multirow{2}*{$\delta\geq 3$}  & \multirow{5}*{\minitab[c]{$\frac{3}{7}>0.428$\\Prop.~\ref{prop:heawood-LD}}} & $\frac{4}{11}-\epsilon>0.363$ &  & \multirow{5}*{\minitab[c]{0.5\\Prop.~\ref{prop:G12-ID}}} & $\frac{5}{11}-\epsilon>0.454$ & \\
&  & Prop.~\ref{prop:example-LD-conn-4/11} &  &  & Prop.~\ref{prop:example-ID-5n/11} & \\
& & & & & & \\[-0.3cm]\cline{3-4}\cline{6-7}\cline{1-1}
&  & & & & & \\[-0.2cm]
\multirow{2}*{cubic}  &   & $\frac{1}{3}$ & $\frac{22}{45}<0.489$ & & $0.4$ & $\frac{31}{45}<0.689$\\
&  & Thm.~\ref{th:LB-Delta}  \cite{S95} & Cor.~\ref{cor:LD} &  & Thm.~\ref{th:LB-Delta} \cite{KCL98} & Cor.~\ref{cor:ID}\\
    \end{tabular}
}
\caption{Upper bounds and largest known ratios (in terms of the graph's order) of location-domination and identifying code numbers in connected graphs of girth at least~$5$ and minimum degree $\delta$.}
\label{tab:bounds}
\end{table}

\noindent\emph{Our results and structure of the paper.} We study the cases where the minimum degree $\delta\in\{2,3\}$, and also the case of cubic graphs. In Section~\ref{sec:UB}, we give upper bounds on parameters $\LD$ and $\M$ for these graph classes, and discuss their tightness by constructing examples with large values of $\LD$ and $\M$ in Section~\ref{sec:cons}. We briefly conclude in Section~\ref{sec:conclu}. A summary of our results is given in Table~\ref{tab:bounds}.  To obtain the upper bounds, we use the technique of building vertex-disjoint paths of the graph, that was introduced by Reed~\cite{R96} for dominating sets and was used in related works, see e.g.~\cite{KS09,XSC06}.

\section{Upper bounds using vertex-disjoint path covers}\label{sec:UB}

This section contains the proofs of our upper bounds. We start with some preliminary tools.

\subsection{Preliminary lemmas and definitions}

Next, we give useful characterizations of locating-dominating sets and identifying codes in graphs of girth~$5$.

\begin{lemma}\label{lemma:girth5-LD}
Let $G$ be a graph of girth at least~5, and let $C$ be a dominating set of $G$. Let $X= \{x \in V(G) \setminus C \; : \; |N(x) \cap C| = 1\}$. Then $C$ is a locating-dominating set of $G$ if and only if there is an injective function $f:X \rightarrow C$ such that $f(x) \in C \cap N(x)$ for all $x \in X$.
\end{lemma}
\begin{proof}
If $C$ is a locating-dominating set, $N(x) \cap C \neq N(y) \cap C$ for each pair $x, y$ of vertices of $X$. Then clearly the function $f:X \rightarrow C$ such that $f(x) = N(x) \cap C$ is injective (if there are two vertices $x,y\in X$ with $y\neq x$ and $f(x)=f(y)$, then $x,y$ would not be separated, a contradiction).

For the sufficiency, suppose that there is an injective function $f:X \rightarrow C$ such that $f(x) \in C \cap N(x)$ for all $x \in X$. Let $Y = V(G) \setminus (C \cup X)$ and let $u, v$ be distinct vertices of $V(G) \setminus C$. If $u, v \in X$, evidently $f(u) \neq f(v)$ and thus $N(u) \cap C \neq N(v) \cap C$. If $u \in X$ and $v \in Y$, then $|N(u) \cap C| = 1$ and $|N(v)\cap C| \ge 2$ and thus $N(u) \cap C\neq N(v)\cap C$. Lastly, if $u, v \in Y$, then $|N(u) \cap C| \ge 2$ and $|N(v) \cap C| \ge 2$. But then $N(u) \cap C \neq N(u) \cap C$ since otherwise there would be a cycle of length~$4$. Thus, $C$ is a locating-dominating set.
\end{proof}

Lemma~\ref{lemma:girth5-LD} means that in a graph of girth~5, the fact that a dominating set is also locating only depends on the vertices that are dominated by exactly one vertex.

The following is a more complicated version of Lemma~\ref{lemma:girth5-LD} for identifying codes. It is a more precise extension of a lemma used by the second author and Perarnau in~\cite{FP11}.

\begin{lemma}\label{lemma:girth5-ID}
Let $G$ be an identifiable graph of girth at least~5. Let $C$ be a dominating set of $G$ and let $C_{\geq 3}$ be the set of vertices of $C$ belonging to a connected component of $G[C]$ of size at least $3$. Then, $C$ is an identifying code of $G$ if and only if the following conditions hold:\\ (i) None of the components of $G[C]$ have size $2$;\\
(ii) For $X = \{x \in V(G)\setminus C \; : \; |N(x) \cap C| = 1\}$, there is an injective function $f:X \rightarrow C$ such that $f(x) \in C_{\geq 3} \cap N(x)$ for all $x \in X$.
\end{lemma}
\begin{proof}
First, assume that $C$ is an identifying code of $G$. Then,
Property~(i) is clear (otherwise the two vertices of some component
$C_i$ of order~2 would not be separated). The proof that Property~(ii) holds is similar as for Lemma~\ref{lemma:girth5-LD} (by letting $f(x) = N(x) \cap C$ for each $x\in X$). Observe that $f(x)\in C_{\geq 3}$, otherwise $x$ and $f(x)$ would not be separated.

For the other side, assume that $C$ is a dominating set fulfilling Properties~(i) and~(ii) and, by contradiction, assume that there are two distinct vertices $x,y$ that are not separated by $C$, i.e. $N[x] \cap C = N[y] \cap C$.

Assume first that $x$ and $y$ are adjacent. As $N[x] \cap C = N[y] \cap C \neq \emptyset$, it follows that $N[x] \cap C = N[y] \cap C\subseteq \{x,y\}$ (since there is no triangle in $G$). If both $x,y$ belong to $C$, $x$ and $y$ induce a component of $G[C]$ of size $2$, a contradiction. Otherwise, exactly one of them belongs to $C$ (say $x$). But then $y$ is only dominated by $x$, which does not belong to $C_{\geq 3}$, a contradiction to Property~(ii).

Thus, $x$ and $y$ are non-adjacent and, since there are no $4$-cycles, $|N(x) \cap N(y)| \le 1$. Hence, there is a vertex $z$ with $N(x) \cap C = N(y) \cap C = \{z\}$. It follows that $x, y \in X$ but $f(x) = z=f(y)$, a contradiction.
\end{proof}

We now define the key concept of vertex-disjoint path cover of a graph, and introduce some related notation.

\begin{definition}
A \emph{vertex-disjoint path cover} (\emph{vdp-cover} for short) of $G$ is a partition of $V(G)$ into sets of vertices, each of them inducing a graph with a Hamiltonian path.

For $0\leq i\leq 4$, a path whose order is congruent to $i$ modulo~$5$ is called an \emph{$(i\bmod 5)$-path}, and a path of order~$j$ is a $j$-path (an empty path is a $0$-path). Given a vdp-cover $\mathcal S$, we will denote by $\mathcal S_i$ the set of $(i\bmod 5)$-paths in $\mathcal S$, and by $\mathcal T_i$, the set of $i$-paths in $\mathcal S$.
\end{definition}

The following result of Reed~\cite{R96} will be used.

\begin{theorem}[\cite{R96}]\label{thm:reed-vdp}
Every connected cubic graph of order $n$ has a vdp-cover with at most $\frac{n}{9}$ sets.
\end{theorem}

\subsection{Locating-dominating sets}\label{sec:LD-upperbounds}

The bound given in the following theorem also follows from a stronger result in a recent paper by Garijo, Gonz\'alez and M\'arquez~\cite{GGM} (see there Proposition 6.6). However, we give an independent proof by a completely different method, which is a good and simple illustration of this technique that will be used several times in this paper.

\begin{theorem}\label{thm:vdp-LD}
Let $G$ be a graph of order $n$, girth at least~5 and minimum degree at least~$2$. Then $\LD(G)\leq\frac{n}{2}$.
\end{theorem}
\begin{proof}
Let $\mathcal S$ be a vdp-cover of $G$ and let $\mathcal{T}_1$ and $\mathcal{T}_3$ be the sets of paths of order~$1$ and~$3$ in $\mathcal{S}$, respectively. Let $\mathcal{S}$ be chosen such that $2|\mathcal T_1|+|\mathcal T_3|$ is minimized. Without loss of generality, we can assume that all paths in $\mathcal S$ have length at most~$5$, since otherwise we can split any longer path into paths of lengths~$2$ or~$5$ without affecting the minimality condition. For each path $P \in \mathcal S$ of length $1 \le r \le 5$, we define an order $P = x_0-x_1- \ldots-x_{r-1}$ with $x_i$ adjacent to $x_{i+1}$ for $0\leq i< r-1$. Let $D$ be the set of vertices containing all vertices of the paths of $\mathcal S$ of odd index (i.e. all $x_1$'s and all $x_3$'s). Note that $D$ clearly dominates all vertices, except possibly the vertices in a $1$-path. Also, we define a function $f$ on all vertices with index~$0$ or~$4$ in the following way. If $P= x_0-x_1-\ldots-x_{r-1}$ is an $r$-path with $2 \le r \le 5$, then $f(x_0) = x_1$ and, if $r = 5$, $f(x_4) = x_3$. According to Lemma \ref{lemma:girth5-LD}, if the end-vertices of the $3$-paths in $\mathcal{S}$ have, besides of their neighbour on the path, a second neighbour in $D$ and if all vertices of the $1$-paths from $\mathcal S$ have two neighbours in $D$, then the restriction of $f$ to the set of $1$-dominated vertices is injective and therefore $D$ is a locating-dominating set.

We will show that every vertex of a $1$-path and every end-vertex of a $3$-path has no neighbour outside of $D$; which by the previous discussion suffices for $D$ being a locating-dominating set. Herefor, we say that a vertex $x$ is a $(p,q)$-vertex if it belongs to a path $P$ of order $p+q+1$ of $\mathcal S$ and the two paths obtained from $P$ by removing $x$ have orders $p$ and $q$. Observe that a $(p,q)$-vertex is the same as a $(q,p)$-vertex. Further, we say that, for fixed~$p$ and~$q$, the $(p,q)$-vertices are \emph{good} if they all belong to $D$, otherwise they are \emph{bad}. Taking into account that $p+q+1 \le 5$, we have the following pairs $(p,q)$ such that $(p,q)$-vertices are bad: $(0,0), (0,1), (0,2), (0,3), (0,4), (1,2)$ and $(2,2)$.

Let $P = x \in \mathcal{T}_1$ be a $1$-path. If $x$ is adjacent to a $(0,q)$-vertex for some $q \in \{0,1,2,3,4\}$, then we can replace the $1$-path and the $(q+1)$-path by a $(q+2)$-path, obtaining in all cases a lower value for the sum $2|\mathcal T_1|+|\mathcal T_3|$, a contradiction. Hence suppose that $x$ is adjacent to either a $(1,2)$-vertex or to a $(2,2)$-vertex. Then we can substitute the $1$-path and the $4$- or $5$-path by a $2$-path and a $3$- or $4$-path, obtaining in each case a lower value for the sum $2|\mathcal T_1|+|\mathcal T_3|$, which is a contradiction. Hence, $x$ has to be adjacent only to good vertices. As $\delta(G) \ge 2$, it follows that $x$ is adjacent to two vertices from $D$. Completely analogous we obtain a contradiction when $P$ is a $3$-path having an end-vertex adjacent to a bad vertex. Altogether, it follows that all vertices not in $D$ have either an assignment via $f$ or two neighbours in $D$. Hence, by Lemma \ref{lemma:girth5-LD}, $D$ is a locating-dominating set. Since each path from $S$ has at most half of its vertices in $D$, we obtain $\LD(G)\leq |D| \leq \frac{n}{2}$.
\end{proof}

Theorem~\ref{thm:vdp-LD} is tight for the cycles $C_6$ and $C_8$, which can easily be seen to have location-domination numbers~$3$ and~$4$, respectively (see also~\cite{S88}). In Proposition~\ref{prop:example-LD-n/2}, we will give a construction of arbitrarily large connected graphs based on copies of $C_6$.

Next, given a vdp-cover $\mathcal S$ of a graph $G$ with girth~5, we will show how to construct a set $D(\mathcal{S})$ and an injective function $f:X\rightarrow D(\mathcal{S})$ (where $X$ is the set of $1$-dominated vertices of $V(G)\setminus D(\mathcal{S})$) meeting the conditions of Lemma~\ref{lemma:girth5-LD}. We will build $D(\mathcal{S})$ by taking roughly two vertices out of five in each path of $\mathcal S$, then adding a few vertices for each path whose length is nonzero modulo~$5$.

\begin{definition}\label{def:D(S)} Let $G$ be a graph of girth at least~$5$ and $\mathcal S$ be a vdp-cover of $G$. Then, the set $D(\mathcal S)$ and the function $f_{D(\mathcal S)}$ are constructed as follows.

For each path $P=x_0-\ldots-x_{p-1}$ in $\mathcal S$, we do the following. Assume that $P\in \mathcal S_i$ ($0\leq i\leq 4$), that is, $p=5k+i$ for some $k\geq 0$. If $k\geq 1$, $D(\mathcal S)$ contains the set $\{x_j\in V(P), j=1,3\bmod 5, j< 5k\}$.

Now, if $k\geq 0$ and $P$ belongs to $\mathcal S\setminus \mathcal S_0$, we add some vertices to $D(\mathcal S)$ according to the following case distinction:

\begin{itemize}
\item If $P\in \mathcal S_1$, we let $D(\mathcal S)$ contain $x_{p-1}$.
\item If $P\in \mathcal S_2$, $D(\mathcal S)$ also contains
  $x_{p-2}$ and $f_{D(\mathcal S)}(x_{p-1})=x_{p-2}$.
\item If $P\in \mathcal S_3$, $D(\mathcal S)$ also contains
  $\{x_{p-3},x_{p-2}\}$ and $f_{D(\mathcal S)}(x_{p-1})=x_{p-2}$.
\item If $P\in \mathcal S_4$, $D(\mathcal S)$ also contains
  $\{x_{p-3},x_{p-1}\}$ and $f_{D(\mathcal S)}(x_{p-4})=x_{p-3}$.
\end{itemize}

To finish the construction of the function $f_{D(\mathcal S)}$, for $j< 5k$, if $x_j\notin D(\mathcal S)$ and $j=0\bmod 5$, $f_{D(\mathcal S)}(x_j)=x_{j+1}$; if $j=4\bmod 5$, $f_{D(\mathcal S)}(x_j)=x_{j-1}$.
\end{definition}

An illustration of Definition~\ref{def:D(S)} is given in Figure~\ref{fig:D(S)}.

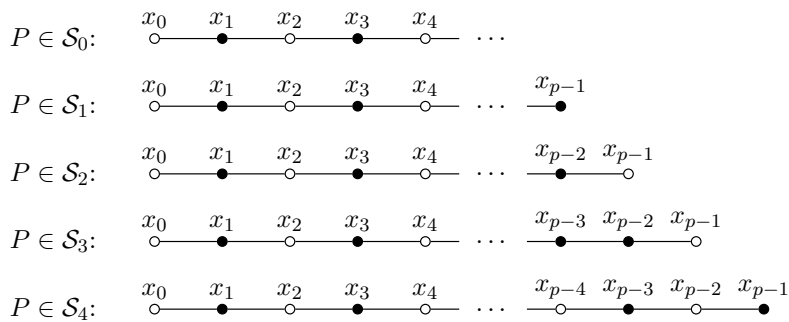
\begin{figure}[!htpb]
\centering
\scalebox{0.9}{\begin{tikzpicture}[join=bevel,inner sep=0.5mm,scale=1.0,line width=0.5pt]

%% S0
\path (0,4) node[draw,shape=circle] (a) {};
\draw (a)+(-1.5,0) node {$P\in\mathcal S_0$:};
\draw (a) node[above=0.1cm] {$x_0$};
\path (1,4) node[draw,shape=circle,fill=black] (b) {};
\draw (b) node[above=0.1cm] {$x_1$};
\path (2,4) node[draw,shape=circle] (c) {};
\draw (c) node[above=0.1cm] {$x_2$};
\path (3,4) node[draw,shape=circle,fill=black] (d) {};
\draw (d) node[above=0.1cm] {$x_3$};
\path (4,4) node[draw,shape=circle] (e) {};
\draw (e) node[above=0.1cm] {$x_4$};
\draw (e)+(1,0) node {\ldots};
\draw (a) -- (b) -- (c) -- (d) -- (e) -- ++(0.5,0);

%% S1
\path (0,3) node[draw,shape=circle] (a) {};
\draw (a)+(-1.5,0) node {$P\in\mathcal S_1$:};
\draw (a) node[above=0.1cm] {$x_0$};
\path (1,3) node[draw,shape=circle,fill=black] (b) {};
\draw (b) node[above=0.1cm] {$x_1$};
\path (2,3) node[draw,shape=circle] (c) {};
\draw (c) node[above=0.1cm] {$x_2$};
\path (3,3) node[draw,shape=circle,fill=black] (d) {};
\draw (d) node[above=0.1cm] {$x_3$};
\path (4,3) node[draw,shape=circle] (e) {};
\draw (e) node[above=0.1cm] {$x_4$};
\draw (e)+(1,0) node {\ldots};
\path (6,3) node[draw,shape=circle,fill=black] (f) {};
\draw (f) node[above=0.1cm] {$x_{p-1}$};
\draw (a) -- (b) -- (c) -- (d) -- (e) -- ++(0.5,0)
 (f) -- ++(-0.5,0);

%% S2
\path (0,2) node[draw,shape=circle] (a) {};
\draw (a)+(-1.5,0) node {$P\in\mathcal S_2$:};
\draw (a) node[above=0.1cm] {$x_0$};
\path (1,2) node[draw,shape=circle,fill=black] (b) {};
\draw (b) node[above=0.1cm] {$x_1$};
\path (2,2) node[draw,shape=circle] (c) {};
\draw (c) node[above=0.1cm] {$x_2$};
\path (3,2) node[draw,shape=circle,fill=black] (d) {};
\draw (d) node[above=0.1cm] {$x_3$};
\path (4,2) node[draw,shape=circle] (e) {};
\draw (e) node[above=0.1cm] {$x_4$};
\draw (e)+(1,0) node {\ldots};
\path (6,2) node[draw,shape=circle,fill=black] (f) {};
\draw (f) node[above=0.1cm] {$x_{p-2}$};
\path (7,2) node[draw,shape=circle] (g) {};
\draw (g) node[above=0.1cm] {$x_{p-1}$};
\draw (a) -- (b) -- (c) -- (d) -- (e) -- ++(0.5,0)
 (g) -- (f) -- ++(-0.5,0);

%% S3
\path (0,1) node[draw,shape=circle] (a) {};
\draw (a)+(-1.5,0) node {$P\in\mathcal S_3$:};
\draw (a) node[above=0.1cm] {$x_0$};
\path (1,1) node[draw,shape=circle,fill=black] (b) {};
\draw (b) node[above=0.1cm] {$x_1$};
\path (2,1) node[draw,shape=circle] (c) {};
\draw (c) node[above=0.1cm] {$x_2$};
\path (3,1) node[draw,shape=circle,fill=black] (d) {};
\draw (d) node[above=0.1cm] {$x_3$};
\path (4,1) node[draw,shape=circle] (e) {};
\draw (e) node[above=0.1cm] {$x_4$};
\draw (e)+(1,0) node {\ldots};
\path (6,1) node[draw,shape=circle,fill=black] (f) {};
\draw (f) node[above=0.1cm] {$x_{p-3}$};
\path (7,1) node[draw,shape=circle,fill=black] (g) {};
\draw (g) node[above=0.1cm] {$x_{p-2}$};
\path (8,1) node[draw,shape=circle] (h) {};
\draw (h) node[above=0.1cm] {$x_{p-1}$};
\draw (a) -- (b) -- (c) -- (d) -- (e) -- ++(0.5,0)
 (h) -- (g) -- (f) -- ++(-0.5,0);

%% S4
\path (0,0) node[draw,shape=circle] (a) {};
\draw (a)+(-1.5,0) node {$P\in\mathcal S_4$:};
\draw (a) node[above=0.1cm] {$x_0$};
\path (1,0) node[draw,shape=circle,fill=black] (b) {};
\draw (b) node[above=0.1cm] {$x_1$};
\path (2,0) node[draw,shape=circle] (c) {};
\draw (c) node[above=0.1cm] {$x_2$};
\path (3,0) node[draw,shape=circle,fill=black] (d) {};
\draw (d) node[above=0.1cm] {$x_3$};
\path (4,0) node[draw,shape=circle] (e) {};
\draw (e) node[above=0.1cm] {$x_4$};
\draw (e)+(1,0) node {\ldots};
\path (6,0) node[draw,shape=circle] (f) {};
\draw (f) node[above=0.1cm] {$x_{p-4}$};
\path (7,0) node[draw,shape=circle,fill=black] (g) {};
\draw (g) node[above=0.1cm] {$x_{p-3}$};
\path (8,0) node[draw,shape=circle] (h) {};
\draw (h) node[above=0.1cm] {$x_{p-2}$};
\path (9,0) node[draw,shape=circle,fill=black] (i) {};
\draw (i) node[above=0.1cm] {$x_{p-1}$};

\draw (a) -- (b) -- (c) -- (d) -- (e) -- ++(0.5,0)
 (i) -- (h) -- (g) -- (f) -- ++(-0.5,0);

\end{tikzpicture}}
\caption{Illustration of set $D(\mathcal S)$.}\label{fig:D(S)}
\end{figure}

\begin{lemma}\label{lemma:D(S)}
Let $G$ be a graph of girth at least~$5$ having a vdp-cover $\mathcal S$. Then $D(\mathcal S)$ is a locating-dominating set of $G$.
\end{lemma}
\begin{proof}
The proof follows from Lemma~\ref{lemma:girth5-LD}; indeed, each vertex $x$ of a path $P\in \mathcal S$ and $x\notin D(\mathcal S)$ that is not $2$-dominated has an image $f_{D(\mathcal S)}(x)\in P$ (and no other such vertex $y$ has $f_{D(\mathcal S)}(x)=f_{D(\mathcal S)}(y)$). It follows that the restriction of $f_{D(\mathcal S)}$ to the set $X$ of $1$-dominated vertices is injective.
\end{proof}

Now, using Theorem~\ref{thm:reed-vdp} and the above construction of the set $D(\mathcal S)$, we can give an improved bound for cubic graphs, based on the following general theorem:

\begin{theorem}\label{thm:vdp-ld-cubic}
Let $G$ be a graph of order~$n$, girth at least~$5$ and having a vdp-cover with $\alpha\cdot n$ paths. Then $\LD(G)\leq\frac{2+4\alpha}{5}n$.
\end{theorem}
\begin{proof}
Let $\mathcal S$ be a vdp-cover of $G$ of size at most $\alpha\cdot n$. We consider the set $D(\mathcal S)$ defined in Definition~\ref{def:D(S)}. By Lemma~\ref{lemma:D(S)}, $D(\mathcal S)$ is a locating-dominating set of $G$. It remains to estimate the size of $D(\mathcal S)$.

For each path $P$ in $\mathcal S_i$ with $5k+i$ vertices ($k\geq 0$) blue, we have added $\frac{2k}{5}$ vertices of $P$ to $D(\mathcal S)$ in the first step of the construction. Then, in the second step, for each path in $\mathcal S_1\cup \mathcal S_2$ and $\mathcal S_3\cup \mathcal S_4$, we have added one and two additional vertices, respectively. So in total we get:

\begin{align*}
|D(\mathcal S)|&\leq\frac{2}{5}(n-|\mathcal S_1|-2|\mathcal S_2|-3|\mathcal S_3|-4|\mathcal S_4|)+|\mathcal S_1|+|\mathcal S_2|+2|\mathcal S_3|+2|\mathcal S_4|\\
 &=\frac{2}{5}n+\frac{3}{5}|\mathcal S_1|+\frac{1}{5}|\mathcal S_2|+\frac{4}{5}|\mathcal S_3|+\frac{2}{5}|\mathcal S_4|\\
&\leq\frac{2}{5}n+\frac{4}{5}|\mathcal S|\\
 &\leq\frac{2+4\alpha}{5}n.
\end{align*}
\end{proof}

We get the following corollary of Theorems~\ref{thm:reed-vdp} and~\ref{thm:vdp-ld-cubic}:

\begin{corollary}\label{cor:LD}
Let $G$ be a connected cubic graph of order~$n$ and girth at least~$5$. Then $\LD(G)\leq\frac{22}{45}n<0.489n$.
\end{corollary}

\subsection{Identifying codes}\label{sec:ID-upperbounds}

The methods used in this subsection are similar to the ones of Subsection~\ref{sec:LD-upperbounds}, but the proofs are more intricate.

\begin{theorem}\label{thm:vdp-ID-delta-2}
Let $G$ be an identifiable graph of order~$n$, girth at least~$5$ and minimum degree~$\delta\geq 2$. Then, $\M(G)\leq\frac{5}{7}n<0.715n$.
\end{theorem}
\begin{proof}
Given a vdp-cover $\mathcal S$ of $G$, let $\mathcal{T}_i$ be the set of paths of order exactly~$i$ of $\mathcal{S}$ (in this proof we do not consider the orders modulo~$5$). We choose $\mathcal{S}$ such that 
\begin{equation}\label{eq:ID-min}
4 |\mathcal{T}_1 \cup \mathcal{T}_4| + 3 |\mathcal{T}_2 \cup \mathcal{T}_3| + 2 |\mathcal{T}_8 \cup \mathcal{T}_9| 
\end{equation}
is minimized.
Let $P \in \mathcal{S}$ be an $r$-path with $r \ge 10$. Then we can replace $P$ by paths of orders $5$, $6$ and $7$ without affecting the minimality of \eqref{eq:ID-min}:
\begin{itemize}
\item If $r \equiv 0  \mod 5$, then we can replace $P$ by $5$-paths.
\item If $r \equiv 1  \mod 5$, then we can replace $P$ by one $6$-path and the remaining part by $5$-paths.
\item If $r \equiv 2  \mod 5$, then we can replace $P$ by one $7$-path and the remaining part by $5$-paths.
\item If $r \equiv 3  \mod 5$, then we can replace $P$ by one $6$-path, one $7$-path, and the remaining part by $5$-paths.
\item If $r \equiv 4  \mod 5$, then we can replace $P$ by two $7$-paths and the remaining part by $5$-paths.
\end{itemize}
Hence, without loss of generality, we can assume that there are no paths of length~$10$ or more in $\mathcal{S}$. 
Now, we define a set $C$ in the following way. For each $r$-path $P = x_0-x_1-\ldots-x_{r-1}$ of $\mathcal{S}$, we add some vertices to $C$ and define a function $f$ according to the following distinction:
\begin{itemize}
\item If $r = 2$, then let $C$ contain $x_1$.
\item If $r = 3$, then let $C$ contain $x_1$ and $x_2$.
\item If $r = 4$, then let $C$ contain $x_0$ and $x_3$; let $f(x_1) = x_0$ and $f(x_2) = x_3$.
\item If $5 \le r \le 7$, then let $C$ contain $x_1, x_2, \ldots x_{r-2}$; let $f(x_0) = x_1$ and $f(x_{r-1}) = x_{r-2}$.
\item If $r = 8$, then let $C$ contain $x_1$, $x_2$, $x_3$, $x_6$ and $x_7$; let $f(x_4) = x_3$ and $f(x_5) = x_6$.
\item If $r = 9$, then let $C$ contain $x_1$, $x_2$, $x_3$, $x_4$, $x_7$ and $x_8$; let $f(x_5) = x_4$ and $f(x_6) = x_7$.
\end{itemize}
An illustration of set $C$ is given in Figure~\ref{fig:C}. We will show that $C$ is an identifying code of $G$.

\begin{figure}[!htpb]
\centering
\scalebox{0.9}{\begin{tikzpicture}[join=bevel,inner sep=0.5mm,scale=1.0,line width=0.5pt]

%%p=1
\path (0,8) node[draw,shape=circle] (a) {};
\draw (a)+(-1.5,0) node {$r = 1$:};

%% p=2
\path (0,7) node[draw,shape=circle] (a) {};
\draw (a)+(-1.5,0) node {$r = 2$:};
\draw (a) node[above=0.1cm] {$x_0$};
\path (1,7) node[draw,shape=circle,fill=black] (b) {};
\draw (b) node[above=0.1cm] {$x_1$};
\draw (a) -- (b);

%% p=3
\path (0,6) node[draw,shape=circle] (a) {};
\draw (a)+(-1.5,0) node {$r=3$:};
\draw (a) node[above=0.1cm] {$x_0$};
\path (1,6) node[draw,shape=circle,fill=black] (b) {};
\draw (b) node[above=0.1cm] {$x_1$};
\path (2,6) node[draw,shape=circle,fill=black] (c) {};
\draw (c) node[above=0.1cm] {$x_2$};
\draw (a) -- (b) -- (c);

%% p=4
\path (0,5) node[draw,shape=circle,fill=black] (a) {};
\draw (a)+(-1.5,0) node {$r=4$:};
\draw (a) node[above=0.1cm] {$x_0$};
\path (1,5) node[draw,shape=circle] (b) {};
\draw (b) node[above=0.1cm] {$x_1$};
\path (2,5) node[draw,shape=circle] (c) {};
\draw (c) node[above=0.1cm] {$x_2$};
\path (3,5) node[draw,shape=circle,fill=black] (d) {};
\draw (d) node[above=0.1cm] {$x_3$};
\draw (a) -- (b) -- (c) -- (d);

%% p=5
\path (0,4) node[draw,shape=circle] (a) {};
\draw (a)+(-1.5,0) node {$r=5$:};
\draw (a) node[above=0.1cm] {$x_0$};
\path (1,4) node[draw,shape=circle,fill=black] (b) {};
\draw (b) node[above=0.1cm] {$x_1$};
\path (2,4) node[draw,shape=circle,fill=black] (c) {};
\draw (c) node[above=0.1cm] {$x_2$};
\path (3,4) node[draw,shape=circle,fill=black] (d) {};
\draw (d) node[above=0.1cm] {$x_3$};
\path (4,4) node[draw,shape=circle] (e) {};
\draw (e) node[above=0.1cm] {$x_4$};
\draw (a) -- (b) -- (c) -- (d) -- (e);

%% p=6
\path (0,3) node[draw,shape=circle] (a) {};
\draw (a)+(-1.5,0) node {$r=6$:};
\draw (a) node[above=0.1cm] {$x_0$};
\path (1,3) node[draw,shape=circle,fill=black] (b) {};
\draw (b) node[above=0.1cm] {$x_1$};
\path (2,3) node[draw,shape=circle,fill=black] (c) {};
\draw (c) node[above=0.1cm] {$x_2$};
\path (3,3) node[draw,shape=circle,fill=black] (d) {};
\draw (d) node[above=0.1cm] {$x_3$};
\path (4,3) node[draw,shape=circle,fill=black] (e) {};
\draw (e) node[above=0.1cm] {$x_4$};
\path (5,3) node[draw,shape=circle] (f) {};
\draw (f) node[above=0.1cm] {$x_5$};
\draw (a) -- (b) -- (c) -- (d) -- (e) -- (f);

%% p=7
\path (0,2) node[draw,shape=circle] (a) {};
\draw (a)+(-1.5,0) node {$r=7$:};
\draw (a) node[above=0.1cm] {$x_0$};
\path (1,2) node[draw,shape=circle,fill=black] (b) {};
\draw (b) node[above=0.1cm] {$x_1$};
\path (2,2) node[draw,shape=circle,fill=black] (c) {};
\draw (c) node[above=0.1cm] {$x_2$};
\path (3,2) node[draw,shape=circle,fill=black] (d) {};
\draw (d) node[above=0.1cm] {$x_3$};
\path (4,2) node[draw,shape=circle,fill=black] (e) {};
\draw (e) node[above=0.1cm] {$x_4$};
\path (5,2) node[draw,shape=circle,fill=black] (f) {};
\draw (f) node[above=0.1cm] {$x_5$};
\path (6,2) node[draw,shape=circle] (g) {};
\draw (g) node[above=0.1cm] {$x_6$};
\draw (a) -- (b) -- (c) -- (d) -- (e) -- (f) -- (g);

%% p=8
\path (0,1) node[draw,shape=circle] (a) {};
\draw (a)+(-1.5,0) node {$r=8$:};
\draw (a) node[above=0.1cm] {$x_0$};
\path (1,1) node[draw,shape=circle,fill=black] (b) {};
\draw (b) node[above=0.1cm] {$x_1$};
\path (2,1) node[draw,shape=circle,fill=black] (c) {};
\draw (c) node[above=0.1cm] {$x_2$};
\path (3,1) node[draw,shape=circle,fill=black] (d) {};
\draw (d) node[above=0.1cm] {$x_3$};
\path (4,1) node[draw,shape=circle] (e) {};
\draw (e) node[above=0.1cm] {$x_4$};
\path (5,1) node[draw,shape=circle] (f) {};
\draw (f) node[above=0.1cm] {$x_5$};
\path (6,1) node[draw,shape=circle,fill=black] (g) {};
\draw (g) node[above=0.1cm] {$x_6$};
\path (7,1) node[draw,shape=circle,fill=black] (h) {};
\draw (h) node[above=0.1cm] {$x_7$};
\draw (a) -- (b) -- (c) -- (d) -- (e) -- (f) -- (g) -- (h);

%% p=9
\path (0,0) node[draw,shape=circle] (a) {};
\draw (a)+(-1.5,0) node {$r=9$:};
\draw (a) node[above=0.1cm] {$x_0$};
\path (1,0) node[draw,shape=circle,fill=black] (b) {};
\draw (b) node[above=0.1cm] {$x_1$};
\path (2,0) node[draw,shape=circle,fill=black] (c) {};
\draw (c) node[above=0.1cm] {$x_2$};
\path (3,0) node[draw,shape=circle,fill=black] (d) {};
\draw (d) node[above=0.1cm] {$x_3$};
\path (4,0) node[draw,shape=circle,fill=black] (e) {};
\draw (e) node[above=0.1cm] {$x_4$};
\path (5,0) node[draw,shape=circle] (f) {};
\draw (f) node[above=0.1cm] {$x_5$};
\path (6,0) node[draw,shape=circle] (g) {};
\draw (g) node[above=0.1cm] {$x_6$};
\path (7,0) node[draw,shape=circle,fill=black] (h) {};
\draw (h) node[above=0.1cm] {$x_7$};
\path (8,0) node[draw,shape=circle,fill=black] (i) {};
\draw (i) node[above=0.1cm] {$x_8$};
\draw (a) -- (b) -- (c) -- (d) -- (e) -- (f) -- (g) -- (h) -- (i);
\end{tikzpicture}}
\caption{Illustration of set $C$ in the proof of Theorem~\ref{thm:vdp-ID-delta-2}.}\label{fig:C}
\end{figure}
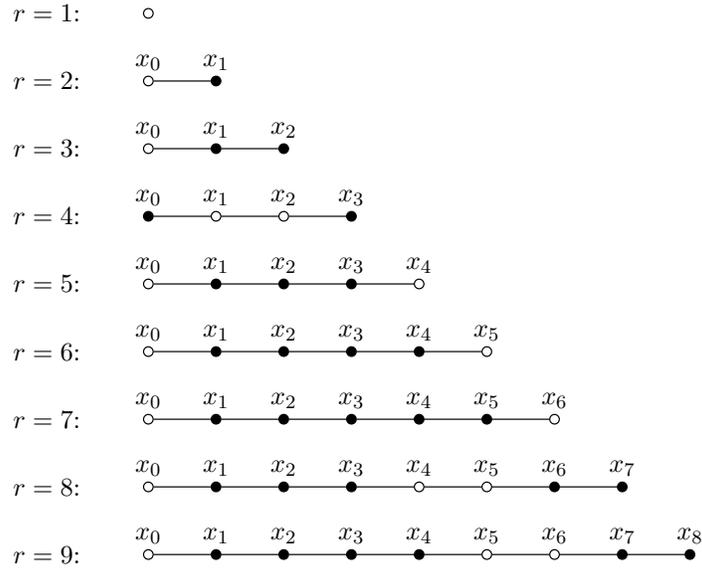

As in the proof of Theorem~\ref{thm:vdp-LD}, we say that a vertex $x$ is a $(p,q)$-vertex if it belongs to a path $P$ of order $p+q+1$ of $\mathcal S$ and the two paths obtained from $P$ by removing $x$ have orders $p$ and $q$. Observe that a $(p,q)$-vertex is the same as a $(q,p)$-vertex. Further, we say that, for fixed $p$ and $q$, the $(p,q)$-vertices are \emph{good} if they all belong to $C$, otherwise they are \emph{bad}. Taking into account that $p+q+1 \le 9$, we have the following set $B$ of pairs $(p,q)$ such that $(p,q)$-vertices are bad:
$$ B = \{ (0,0), (0,1), (0,2), (0,4), (0,5), (0,6), (0,7), (0,8), (1,2), (3,4), (2,5), (2,6), (3,5) \}.$$

Now we will prove the following claims.

\begin{claim}\label{claim:P_8&P_9}
For a path $P \in \mathcal{S}$ of order $r\in\{8,9\}$, we can assume that the end-vertex $x_{r-1}$, which belongs to $C$, has either a second neighbour in $P$ contained in $C$ (i.e. different from its predecessor $x_{r-2}$ in $P$) or it has a neighbour outside $P$.
\end{claim}
Let $r=8$ and $P = x_0-x_1-\ldots-x_7$ and, following the construction of $C$, we have $x_1,x_2,x_3,x_6, x_7 \in C$. By contradiction, suppose that $x_7$ is not adjacent to any of $x_1,x_2,x_3$. Suppose also that $x_7$ has no neighbour outside $P$. Since $G$ has girth at least~$5$, $x_7$ is neither adjacent to $x_4$ nor to $x_5$. Hence, as $\delta \ge 2$, $x_7$ has to be adjacent to $x_0$. Now, either $G = C_8$ or one of the vertices from $P$ has one neighbour outside $P$. In the first case, an independent set of size~$4$ is an identifying code of $G=C_8$ and satisfies the desired bound. Hence we may assume that $G \neq C_8$ and thus there is a vertex from $P$ having a neighbour outside $P$. In this case, we may reorder the vertices along the cycle such that $x_7$ has one neighbour outside $P$. Hence, Claim~\ref{claim:P_8&P_9} follows for $r=8$. The same argument can be used to prove the case $r=9$.

\begin{claim}\label{claim:r-paths}
Let  $r \in \{1, 2, 3, 4, 8, 9\}$ and let $x$ be an end-vertex of an $r$-path $P \in \mathcal{S}$. Then all neighbours of $x$ outside $P$ are good vertices.
\end{claim}

Suppose that, for some $r \in \{1, 2, 3, 4, 8, 9\}$, there is an end-vertex of an $r$-path $P$ which is adjacent to a $(p,q)$-vertex in $P'\in \mathcal S$ with $P\neq P'$ and $(p,q)\in B$. Note that we can replace $P$ and $P'$ by either an $(r+p+1)$-path and a $q$-path or by a $p$-path and an $(r+q+1)$-path.\footnote{Whenever we consider a new $s$-path with $s\geq 10$, we implicitely assume that, as done in the beginning of the proof, it is cut into smaller paths.} We will see that, in each case, we obtain a vdp-cover which contradicts the minimality of~\eqref{eq:ID-min}. If $p = 0$, then we can join the $r$-path together with the $(q+1)$-path obtaining an $(r+q+1)$-path. This gives in all cases a lower value for the sum~\eqref{eq:ID-min}, which is a contradiction. Hence we can suppose that $(p, q) \in \{(1,2), (3,4), (2,5), (2,6), (3,5) \}$.  When $r = 4$ and $(p,q)$ is arbitrary or when $r= 2$ and $(p,q) = (1,2)$, we can replace the $r$- and the $(p+q+1)$-path by an $(r+q+1)$- and a $p$-path and we obtain in all cases a lower value for~\eqref{eq:ID-min}. For $r \in \{1,3,8,9\}$ and $(p,q)$ is arbitrary or $r=2$ and $(p,q) \neq (1,2)$, we can replace the $r$- and the $(p+q+1)$-paths by an $(r+p+1)$- and a $q$-path and we obtain always a lower value for~\eqref{eq:ID-min}. Since we obtain in all cases a contradiction to the minimality of~\eqref{eq:ID-min}, it follows that, for $r =1, 2, 3, 4, 8, 9$, every end-vertex of an $r$-path is adjacent to a good vertex, proving Claim~\ref{claim:r-paths}.

\begin{claim}\label{claim:1-paths}
Every vertex from a $1$-path is adjacent to two vertices of $C$.
\end{claim}
As $\delta \ge 2$ and since by Claim~\ref{claim:r-paths}, the vertex of a $1$-path cannot be adjacent to a bad vertex, then it has to be adjacent to at least two good vertices, proving Claim~\ref{claim:1-paths}.

\begin{claim}\label{claim:components}
There are no components of $G[C]$ that have size at most~$2$.
\end{claim}

Since the girth of $G$ is at least~$5$ and $\delta \ge 2$, all end-vertices of a $2$-, $3$- or $4$-path $P$ have a neighbour outside $P$. By Claim~\ref{claim:r-paths}, these neighbours have to be good vertices. Hence, there are no $1$-components in $G[C]$. On the other side, if $P$ is an $8$- or a $9$-path, Claim~\ref{claim:P_8&P_9} implies that the end-vertices of $P$ have either a further neighbour in $P$ belonging to $C$ or they have a neighbour outside $P$, which, by Claim~\ref{claim:r-paths}, is a good vertex. Thus, the only possibilities to have $2$-components in $G[C]$ are given when two $2$-paths or one $2$-path and one $4$-path or two $4$-paths are connected through their good end-vertices. In these cases we could transform them into a $4$-path, a $6$-path or an $8$-path which would contribute less to the sum~\eqref{eq:ID-min} than the original paths, which is a contradiction, proving Claim~\ref{claim:components}.

Hence, by Claims~\ref{claim:P_8&P_9}, \ref{claim:r-paths}, \ref{claim:1-paths} and \ref{claim:components} and by the construction of the function $f$, $C$ fulfils the conditions of Lemma~\ref{lemma:girth5-ID}, which certifies that it is an identifying code. Since at most $\frac{5}{7}|P|$ vertices from every path $P \in \mathcal{S}$ belong to $C$, $C$ is an identifying code of $G$ of cardinality at most $\frac{5}{7}n$.
\end{proof}

Theorem~\ref{thm:vdp-ID-delta-2} is tight for the cycle $C_7$, which can easily be seen to have identifying code number~$5$ (see also~\cite{BCHL04}).

As for locating-dominating sets, given a vdp-cover $\mathcal S$ of a graph $G$ with girth~$5$, we define a set $C(\mathcal S)$ and a function $f_{C(\mathcal S)}$ as follows.

\begin{definition}\label{def:C(S)}
Let $G$ be a graph of girth at least~5 and $\mathcal S$ be a vdp-cover of $G$. Then, the set $C(\mathcal S)$ and the function $f=f_{C(\mathcal S)}$ are constructed as follows.

For each path $P=x_0-\ldots-x_{p-1}$ of $\mathcal S$, we do the following. Assume that $P\in \mathcal S_i$ ($0\leq i\leq 4$), that is, $p=5k+i$ for some $k\geq 0$. If $k\geq 1$, $C(\mathcal S)$ contains the set $\{x_j\in V(P), j=1,2,3\bmod 5, j< 5k\}$.

Now, for $k\geq 0$, if $P$ belongs to $\mathcal S\setminus \mathcal S_0$, we add some vertices to $C(\mathcal S)$ according to the following case distinction:

\begin{itemize}
\item If $P\in \mathcal S_1$ and $k\geq 1$, we let
  $C(\mathcal S)$ contain $x_{p-2}$ and $f(x_{p-1})=x_{p-2}$. If $k=0$, $C(\mathcal S)$ contains $x_0$.
\item If $P\in \mathcal S_2$ and $k\geq 1$, $C(\mathcal S)$ also contains
  $\{x_{p-3},x_{p-2}\}$ and $f(x_{p-1})=x_{p-2}$. If $k=0$, $C(\mathcal S)$ contains
  $\{x_0,x_1\}$.
\item If $P\in \mathcal S_3$ and $k\geq 1$, $C(\mathcal S)$ also contains
  $\{x_{p-3},x_{p-2},x_{p-1}\}$. If $k=0$, $C(\mathcal S)$ contains
  $\{x_0,x_1,x_2\}$.
\item If $P\in \mathcal S_4$ and $k\geq 1$, $C(\mathcal S)$ also contains
  $\{x_{p-4},x_{p-3},x_{p-2}\}$ and $f(x_{p-1})=x_{p-2}$. If $k=0$, $C(\mathcal S)$ contains
  $\{x_0,x_1,x_2\}$ and $f(x_{3})=x_{2}$.
\end{itemize}

To finish the construction of the function $f$, for $j< 5k$, if $x_j\notin C(\mathcal S)$ and $j=0\bmod 5$, $f(x_j)=x_{j+1}$; if $j=4\bmod 5$, $f(x_j)=x_{j-1}$. Note that each vertex $x\in P$ of $V(G)\setminus C(\mathcal S)$ has an image $f(x)$ belonging to $P$.
\end{definition}

An illustration of Definition~\ref{def:C(S)} is given in Figure~\ref{fig:C(S)}.

\begin{figure}[!htpb]
\centering
\scalebox{0.9}{\begin{tikzpicture}[join=bevel,inner sep=0.5mm,scale=1.0,line width=0.5pt]

%% S0
\path (0,4) node[draw,shape=circle] (a) {};
\draw (a)+(-1.5,0) node {$P\in\mathcal S_0$:};
\draw (a) node[above=0.1cm] {$x_0$};
\path (1,4) node[draw,shape=circle,fill=black] (b) {};
\draw (b) node[above=0.1cm] {$x_1$};
\path (2,4) node[draw,shape=circle,fill=black] (c) {};
\draw (c) node[above=0.1cm] {$x_2$};
\path (3,4) node[draw,shape=circle,fill=black] (d) {};
\draw (d) node[above=0.1cm] {$x_3$};
\path (4,4) node[draw,shape=circle] (e) {};
\draw (e) node[above=0.1cm] {$x_4$};
\draw (e)+(1,0) node {\ldots};
\path (6,4) node[draw,shape=circle] (a2) {};
\draw (a2) node[above=0.1cm] {$x_{p-5}$};
\path (7,4) node[draw,shape=circle,fill=black] (b2) {};
\draw (b2) node[above=0.1cm] {$x_{p-4}$};
\path (8,4) node[draw,shape=circle,fill=black] (c2) {};
\draw (c2) node[above=0.1cm] {$x_{p-3}$};
\path (9,4) node[draw,shape=circle,fill=black] (d2) {};
\draw (d2) node[above=0.1cm] {$x_{p-2}$};
\path (10,4) node[draw,shape=circle] (e2) {};
\draw (e2) node[above=0.1cm] {$x_{p-1}$};

\draw (a) -- (b) -- (c) -- (d) -- (e) -- ++(0.5,0)
      (e2) -- (d2) -- (c2) -- (b2) -- (a2) -- ++(-0.5,0);
%% S1
\path (0,3) node[draw,shape=circle] (a) {};
\draw (a)+(-1.5,0) node {$P\in\mathcal S_1$:};
\draw (a) node[above=0.1cm] {$x_0$};
\path (1,3) node[draw,shape=circle,fill=black] (b) {};
\draw (b) node[above=0.1cm] {$x_1$};
\path (2,3) node[draw,shape=circle,fill=black] (c) {};
\draw (c) node[above=0.1cm] {$x_2$};
\path (3,3) node[draw,shape=circle,fill=black] (d) {};
\draw (d) node[above=0.1cm] {$x_3$};
\path (4,3) node[draw,shape=circle] (e) {};
\draw (e) node[above=0.1cm] {$x_4$};
\draw (e)+(1,0) node {\ldots};
\path (6,3) node[draw,shape=circle] (a2) {};
\draw (a2) node[above=0.1cm] {$x_{p-6}$};
\path (7,3) node[draw,shape=circle,fill=black] (b2) {};
\draw (b2) node[above=0.1cm] {$x_{p-5}$};
\path (8,3) node[draw,shape=circle,fill=black] (c2) {};
\draw (c2) node[above=0.1cm] {$x_{p-4}$};
\path (9,3) node[draw,shape=circle,fill=black] (d2) {};
\draw (d2) node[above=0.1cm] {$x_{p-3}$};
\path (10,3) node[draw,shape=circle,fill=black] (e2) {};
\draw (e2) node[above=0.1cm] {$x_{p-2}$};

\path (11,3) node[draw,shape=circle] (f) {};
\draw (f) node[above=0.1cm] {$x_{p-1}$};
\draw (a) -- (b) -- (c) -- (d) -- (e) -- ++(0.5,0)
      (f) -- (e2) -- (d2) -- (c2) -- (b2) -- (a2) -- ++(-0.5,0);

%% S2
\path (0,2) node[draw,shape=circle] (a) {};
\draw (a)+(-1.5,0) node {$P\in\mathcal S_2$:};
\draw (a) node[above=0.1cm] {$x_0$};
\path (1,2) node[draw,shape=circle,fill=black] (b) {};
\draw (b) node[above=0.1cm] {$x_1$};
\path (2,2) node[draw,shape=circle,fill=black] (c) {};
\draw (c) node[above=0.1cm] {$x_2$};
\path (3,2) node[draw,shape=circle,fill=black] (d) {};
\draw (d) node[above=0.1cm] {$x_3$};
\path (4,2) node[draw,shape=circle] (e) {};
\draw (e) node[above=0.1cm] {$x_4$};
\draw (e)+(1,0) node {\ldots};
\path (6,2) node[draw,shape=circle] (a2) {};
\draw (a2) node[above=0.1cm] {$x_{p-7}$};
\path (7,2) node[draw,shape=circle,fill=black] (b2) {};
\draw (b2) node[above=0.1cm] {$x_{p-6}$};
\path (8,2) node[draw,shape=circle,fill=black] (c2) {};
\draw (c2) node[above=0.1cm] {$x_{p-5}$};
\path (9,2) node[draw,shape=circle,fill=black] (d2) {};
\draw (d2) node[above=0.1cm] {$x_{p-4}$};
\path (10,2) node[draw,shape=circle,fill=black] (e2) {};
\draw (e2) node[above=0.1cm] {$x_{p-3}$};

\path (11,2) node[draw,shape=circle,fill=black] (f) {};
\draw (f) node[above=0.1cm] {$x_{p-2}$};
\path (12,2) node[draw,shape=circle] (g) {};
\draw (g) node[above=0.1cm] {$x_{p-1}$};
\draw (a) -- (b) -- (c) -- (d) -- (e) -- ++(0.5,0)
 (g) -- (f) -- (e2) -- (d2) -- (c2) -- (b2) -- (a2) --  ++(-0.5,0);

%% S3
\path (0,1) node[draw,shape=circle] (a) {};
\draw (a)+(-1.5,0) node {$P\in\mathcal S_3$:};
\draw (a) node[above=0.1cm] {$x_0$};
\path (1,1) node[draw,shape=circle,fill=black] (b) {};
\draw (b) node[above=0.1cm] {$x_1$};
\path (2,1) node[draw,shape=circle,fill=black] (c) {};
\draw (c) node[above=0.1cm] {$x_2$};
\path (3,1) node[draw,shape=circle,fill=black] (d) {};
\draw (d) node[above=0.1cm] {$x_3$};
\path (4,1) node[draw,shape=circle] (e) {};
\draw (e) node[above=0.1cm] {$x_4$};
\draw (e)+(1,0) node {\ldots};
\path (6,1) node[draw,shape=circle] (a2) {};
\draw (a2) node[above=0.1cm] {$x_{p-8}$};
\path (7,1) node[draw,shape=circle,fill=black] (b2) {};
\draw (b2) node[above=0.1cm] {$x_{p-7}$};
\path (8,1) node[draw,shape=circle,fill=black] (c2) {};
\draw (c2) node[above=0.1cm] {$x_{p-6}$};
\path (9,1) node[draw,shape=circle,fill=black] (d2) {};
\draw (d2) node[above=0.1cm] {$x_{p-5}$};
\path (10,1) node[draw,shape=circle] (e2) {};
\draw (e2) node[above=0.1cm] {$x_{p-4}$};

\path (11,1) node[draw,shape=circle,fill=black] (f) {};
\draw (f) node[above=0.1cm] {$x_{p-3}$};
\path (12,1) node[draw,shape=circle,fill=black] (g) {};
\draw (g) node[above=0.1cm] {$x_{p-2}$};
\path (13,1) node[draw,shape=circle,fill=black] (h) {};
\draw (h) node[above=0.1cm] {$x_{p-1}$};
\draw (a) -- (b) -- (c) -- (d) -- (e) -- ++(0.5,0)
 (h) -- (g) -- (f) -- (e2) -- (d2) -- (c2) -- (b2) -- (a2) -- ++(-0.5,0);

%% S4
\path (0,0) node[draw,shape=circle] (a) {};
\draw (a)+(-1.5,0) node {$P\in\mathcal S_4$:};
\draw (a) node[above=0.1cm] {$x_0$};
\path (1,0) node[draw,shape=circle,fill=black] (b) {};
\draw (b) node[above=0.1cm] {$x_1$};
\path (2,0) node[draw,shape=circle,fill=black] (c) {};
\draw (c) node[above=0.1cm] {$x_2$};
\path (3,0) node[draw,shape=circle,fill=black] (d) {};
\draw (d) node[above=0.1cm] {$x_3$};
\path (4,0) node[draw,shape=circle] (e) {};
\draw (e) node[above=0.1cm] {$x_4$};
\draw (e)+(1,0) node {\ldots};
\path (6,0) node[draw,shape=circle] (a2) {};
\draw (a2) node[above=0.1cm] {$x_{p-9}$};
\path (7,0) node[draw,shape=circle,fill=black] (b2) {};
\draw (b2) node[above=0.1cm] {$x_{p-8}$};
\path (8,0) node[draw,shape=circle,fill=black] (c2) {};
\draw (c2) node[above=0.1cm] {$x_{p-7}$};
\path (9,0) node[draw,shape=circle,fill=black] (d2) {};
\draw (d2) node[above=0.1cm] {$x_{p-6}$};
\path (10,0) node[draw,shape=circle] (e2) {};
\draw (e2) node[above=0.1cm] {$x_{p-5}$};

\path (11,0) node[draw,shape=circle,fill=black] (f) {};
\draw (f) node[above=0.1cm] {$x_{p-4}$};
\path (12,0) node[draw,shape=circle,fill=black] (g) {};
\draw (g) node[above=0.1cm] {$x_{p-3}$};
\path (13,0) node[draw,shape=circle,fill=black] (h) {};
\draw (h) node[above=0.1cm] {$x_{p-2}$};
\path (14,0) node[draw,shape=circle] (i) {};
\draw (i) node[above=0.1cm] {$x_{p-1}$};

\draw (a) -- (b) -- (c) -- (d) -- (e) -- ++(0.5,0)
 (i) -- (h) -- (g) -- (f) -- (e2) -- (d2) -- (c2) -- (b2) -- (a2) -- ++(-0.5,0);

\end{tikzpicture}}
\caption{Illustration of set $C(\mathcal S)$.}\label{fig:C(S)}
\end{figure}
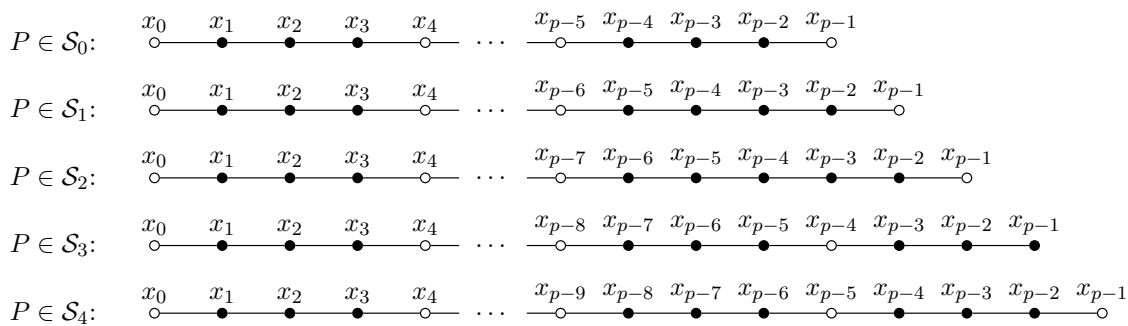

\begin{lemma}\label{lemma:C(S)}
Let $G$ be an identifiable graph of girth at least~$5$ having a vdp-cover $\mathcal S$. Then $C(\mathcal S)$ is a dominating set, and all pairs of vertices are separated, except possibly pairs $x,y$ of vertices such that $x-y$ forms a path of $\mathcal S$.
\end{lemma}
\begin{proof}
The proof follows from Lemma~\ref{lemma:girth5-ID}; indeed, each vertex $x$ of $V(G)\setminus C(\mathcal S)$, with $x\in P$ and $P\in\mathcal S$, that is not $2$-dominated has an image $f_{C(\mathcal S)}(x)\in P$, the restriction of $f_{C(\mathcal S)}$ to the set of $1$-dominated vertices is injective, and the only potentially isolated vertices in $C(\mathcal S)$ are vertices $v$ belonging to a path of $\mathcal S$ of order~$1$ (hence by construction no vertex $x$ has $f_{C(\mathcal S)}(x)=v$).
\end{proof}

Similarly to Theorem~\ref{thm:vdp-ld-cubic} for locating-dominating sets, we have the following generic theorem:

\begin{theorem}\label{thm:vdp-id-cubic}
Let $G$ be an identifiable graph of order $n$, girth at least~$5$ and having a vdp-cover with $\alpha\cdot n$ paths. Then $\M(G)\leq\frac{3+4\alpha}{5}n$.
\end{theorem}
\begin{proof}
Let $\mathcal S$ be the vdp-cover of $G$. The idea is to construct a set $C$ and an injective function $f:X\rightarrow C$ (where $X$ is the set of $1$-dominated vertices of $V(G)\setminus C$) meeting the conditions of Lemma~\ref{lemma:girth5-ID}. We will build $C$ by taking roughly three vertices out of five in each path of $\mathcal S$, then adding a few vertices for each path whose length is nonzero modulo five, and finally performing a few local modifications.

\vspace{0.3cm}\noindent\textit{Step 1: Constructing an initial pseudo-code.} We construct $C=C(\mathcal S)$ and $f=f_{C(\mathcal S)}$ by the procedure described in Definition~\ref{def:C(S)}.

\vspace{0.3cm}
\noindent\textit{Step 2: Taking care of components of $G[C]$ of order 2.}  By Lemma~\ref{lemma:C(S)}, all conditions of Lemma~\ref{lemma:girth5-ID} (where we consider the restriction of $f$ to the set $X$ of $1$-dominated vertices) are fulfilled, except for Property~(i): there might be some paths in $\mathcal S_2$ of order exactly~$2$ and forming a connected component of $G[C]$ (second item of our case distinction). Let $P$ be such a path, and $V(P)=\{x_0,x_1\}$. Then, since $G$ is identifiable, one of $x_0,x_1$ (say $x_1$) has a neighbour $y$, and since $P$ is a connected component in $G[C]$, $y\notin C$. By the construction of $C$, $y$ belongs to a path and is adjacent to vertex $f(y)$ in $C$. We perform the following modification: remove $x_0$ from $C$, put $y$ instead, and let $f(x_0)=x_1$. It is clear that repeating this for each such case, we get rid of all components of order~$2$ in $G[C]$.

Now, all conditions of Lemma~\ref{lemma:girth5-ID} are fulfilled, hence
$C$ is an identifying code of $G$.

\vspace{0.3cm}
\noindent\textit{Step 3: Saving one vertex for each path of $\mathcal S_3$.}

We consider all paths in $\mathcal S_3$ one by one, in an arbitrary order. For each such path $P$ with $V(P)=\{x_0,\ldots,x_{p-1}\}$ ($p=5k+3$ for some $k\geq 0$), we remove $x_{p-3},x_{p-2},x_{p-1}$ from $C$. We now distinguish some cases.

If $x_i\in\{x_{p-2},x_{p-1}\}$ has a neighbour in $C$, then, we add both $x_{p-2},x_{p-1}$ to $C$ and let $f(x_{p-3})=x_{p-2}$. Similarly, if $x_{p-3}$ has a neighbour in $C$, we add both $x_{p-3},x_{p-2}$ to $C$ and let $f(x_{p-1})=x_{p-2}$. Note that in both cases, the two new code-vertices are now part of a component of $G[C]$ of order at least~$3$, hence all conditions of Lemma~\ref{lemma:girth5-ID} are preserved.

If none of $x_{p-3},x_{p-2},x_{p-1}$ have a neighbour in $C$, we add $x_{p-3}$ and $x_{p-1}$ to $C$. Note that $x_{p-2}$ is now $2$-dominated, hence all conditions of Lemma~\ref{lemma:girth5-ID} are again preserved.

Repeating this at every step, $C$ is still an identifying code, and we have decreased the size of $C$ by $|\mathcal S_3|$.

\vspace{0.3cm}
\noindent\textit{Step 4: Estimating the size of the code.} It remains to compute the size of $C$.

For each path $P$ in $\mathcal S_i$ with $5k+i$ vertices, we have added $\frac{3k}{5}$ vertices of $P$ to $C$ in the first phase of the construction of $C(\mathcal S)$ (Definition~\ref{def:C(S)}). Then, in the second phase of the construction of Definition~\ref{def:C(S)}, for each path in $\mathcal S_1,\mathcal S_2,\mathcal S_3,\mathcal S_4$, we have added one, two, three and three additional vertices, respectively. In Steps~2 and~3, we did not change the size of $C$, but in Step~4, we were able to reduce the size of $C$ by one for each path in $\mathcal S_3$. So in total we have:

\begin{align*}
|C|&\leq\frac{3}{5}(n-|\mathcal S_1|-2|\mathcal S_2|-3|\mathcal S_3|-4|\mathcal S_4|)+|\mathcal S_1|+2|\mathcal S_2|+3|\mathcal S_3|+3|\mathcal S_4|-|\mathcal S_3|\\
 &=\frac{3}{5}n+\frac{2}{5}|\mathcal S_1|+\frac{4}{5}|\mathcal S_2|+\frac{1}{5}|\mathcal S_3|+\frac{3}{5}|\mathcal S_4|\\
&\leq\frac{3}{5}n+\frac{4}{5}|\mathcal S|\\
 &\leq\frac{3+4\alpha}{5}n.
\end{align*}
\end{proof}

We get the following improvement for cubic graphs, a corollary of Theorems~\ref{thm:reed-vdp} and~\ref{thm:vdp-id-cubic}:

\begin{corollary}\label{cor:ID}
Let $G$ be a connected cubic identifiable graph of order $n$ and girth at least~5. Then $\M(G)\leq\frac{31}{45}n<0.689n$.
\end{corollary}

\section{Constructions}\label{sec:cons}

In this section, we provide constructions of connected graphs with girth at least~$5$ and large location-domination or identifying code number. First of all, the following result is a lower bound on parameters $\LD$ and $\M$ depending on the maximum degree $\Delta$ of a graph. It will be useful since it also applies to $\Delta$-regular graphs.

\begin{theorem}[\cite{F,KCL98,S95}]\label{th:LB-Delta}
Let $G$ be a graph of order $n$ and maximum degree $\Delta$. Then $\LD(G)\geq\frac{2n}{\Delta+3}$. If $G$ is identifiable, then $\M(G)\geq\frac{2n}{\Delta+2}$, and any identifying code of this size is an independent $2$-dominating set whose vertices all have degree
$\Delta$ in $G$.
\end{theorem}

We remark that the last part of the statement is not very difficult to obtain from the proof of the bound; a proof is available in the first author's PhD thesis~\cite[Section 4.1]{F}.

\subsection{Generic constructions}

We now define constructions based on the Petersen graph that will be used later on.

\begin{definition}\label{def:petersens} Denote by $P_{10}$ the Petersen graph with $V(P_{10})=\{0,1,2,3,4,5,6,7,8,9\}$, and $0-1-2-\ldots-9$ one of its Hamiltonian paths, such that vertices $1$ and $9$ are adjacent. Let $P_{11}$ be the graph obtained from $P_{10}$ by subdividing once the edge $\{0,1\}$, and calling the new vertex $x$.
\end{definition}

The graphs of Definition~\ref{def:petersens} are illustrated in Figure~\ref{fig:petersens}.

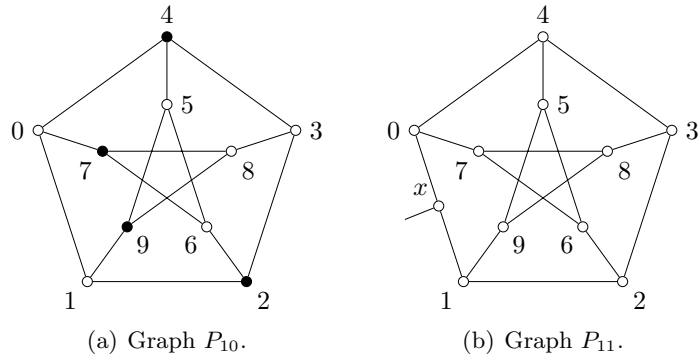
\begin{figure}[!ht]
\begin{center}
\subfigure[Graph $P_{10}$.]{\scalebox{0.9}{\begin{tikzpicture}[join=bevel,inner sep=0.5mm]
\node[graphnode](8) at (18:1) {};
\draw (8) node[below right=0.1cm] {$8$};
\node[graphnode](5) at (90:1) {};
\draw (5) node[right=0.15cm] {$5$};
\node[graphnode,fill](7) at (162:1) {};
\draw (7) node[below left=0.1cm] {$7$};
\node[graphnode,fill](9) at (234:1) {};
\draw (9) node[below right=0.08cm] {$9$};
\node[graphnode](6) at (306:1) {};
\draw (6) node[below left=0.08cm] {$6$};
\node[graphnode](3) at (18:2) {};
\draw (3) node[right=0.15cm] {$3$};
\node[graphnode,fill](4) at (90:2) {};
\draw (4) node[above=0.15cm] {$4$};
\node[graphnode](0) at (162:2) {};
\draw (0) node[left=0.15cm] {$0$};
\node[graphnode](1) at (234:2) {};
\draw (1) node[below left=0.1cm] {$1$};
\node[graphnode,fill](2) at (306:2) {};
\draw (2) node[below right=0.1cm] {$2$};
\node (-6,0) {};
\node (6,0) {};
\draw[-] (7)--(6)--(5)--(9)--(8)--(7);
\draw[-] (3)--(4)--(0)--(1)--(2)--(3);
\draw[-] (7)--(0);
\draw[-] (5)--(4);
\draw[-] (6)--(2);
\draw[-] (9)--(1);
\draw[-] (8)--(3);
\end{tikzpicture}}}\qquad
\subfigure[Graph $P_{11}$.]{\scalebox{0.9}{\begin{tikzpicture}[join=bevel,inner sep=0.5mm]
\node[graphnode](8) at (18:1) {};
\draw (8) node[below right=0.1cm] {$8$};
\node[graphnode](5) at (90:1) {};
\draw (5) node[right=0.15cm] {$5$};
\node[graphnode](7) at (162:1) {};
\draw (7) node[below left=0.1cm] {$7$};
\node[graphnode](9) at (234:1) {};
\draw (9) node[below right=0.08cm] {$9$};
\node[graphnode](6) at (306:1) {};
\draw (6) node[below left=0.08cm] {$6$};
\node[graphnode](3) at (18:2) {};
\draw (3) node[right=0.15cm] {$3$};
\node[graphnode](4) at (90:2) {};
\draw (4) node[above=0.15cm] {$4$};
\node[graphnode](0) at (162:2) {};
\draw (0) node[left=0.15cm] {$0$};
\node[graphnode](x) at (198:1.62) {};
\draw (x) node[above left=0.1cm] {$x$};
\node[graphnode](1) at (234:2) {};
\draw (1) node[below left=0.1cm] {$1$};
\node[graphnode](2) at (306:2) {};
\draw (2) node[below right=0.1cm] {$2$};

\draw[-] (7)--(6)--(5)--(9)--(8)--(7);
\draw[-] (3)--(4)--(0) -- (x) -- (1)--(2)--(3);
\draw[-] (7)--(0);
\draw[-] (5)--(4);
\draw[-] (6)--(2);
\draw[-] (9)--(1);
\draw[-] (8)--(3);
\draw[-] (x)+(-0.5,-0.2)--(x);
\end{tikzpicture}}}
\end{center}
\caption{\label{fig:petersens} The Petersen graph $P_{10}$ and its modification $P_{11}$. The black vertices form an optimal identifying code and locating-dominating set of $P_{10}$.}
\end{figure}

\begin{definition}\label{def:G_11^k} 
For any $k\geq 2$, let $G_{11}^k$ be the graph formed by a vertex $y$
connected to $k$ copies of $P_{11}$ (each attached via vertex $x$).
\end{definition}

The graph $G_{11}^k$ is illustrated in Figure~\ref{fig:P11's}.

\begin{figure}[!ht]
\begin{center}
\subfigure[An optimal locating-dominating set of $G_{11}^k$.]{\label{fig:P11's-LD}\scalebox{0.8}{\begin{tikzpicture}[join=bevel,inner sep=0.5mm]
\node[graphnode](y) at (2.5,3) {};
\draw (y) node[above=0.15cm] {$y$};
\node[graphnode](8) at (18:1) {};
\path (8)+(5,0) node[graphnode,fill] (81) {};
\node[graphnode,fill](5) at (90:1) {};
\path (5)+(5,0) node[graphnode,fill] (51) {};
\node[graphnode,fill](7) at (162:1) {};
\path (7)+(5,0) node[graphnode] (71) {};
\node[graphnode](9) at (234:1) {};
\path (9)+(5,0) node[graphnode] (91) {};
\node[graphnode](6) at (306:1) {};
\path (6)+(5,0) node[graphnode] (61) {};
\node[graphnode](3) at (18:2) {};
\path (3)+(5,0) node[graphnode] (31) {};
\node[graphnode](4) at (90:2) {};
\path (4)+(5,0) node[graphnode] (41) {};
\node[graphnode](0) at (162:2) {};
\path (0)+(5,0) node[graphnode] (01) {};
\node[graphnode,fill](x) at (54:1.62) {};
\path (x)+(3.1,0) node[graphnode,fill] (x1) {};
\node[graphnode,fill](1) at (234:2) {};
\path (1)+(5,0) node[graphnode] (11) {};
\node[graphnode](2) at (306:2) {};
\path (2)+(5,0) node[graphnode,fill] (21) {};

\path (2.5,1.5) node  {$\ldots$};

\draw[-] (7)--(6)--(5)--(9)--(8)--(7);
\draw[-] (3)--(x)--(4)--(0) -- (1)--(2)--(3);
\draw[-] (7)--(0);
\draw[-] (5)--(4);
\draw[-] (6)--(2);
\draw[-] (9)--(1);
\draw[-] (8)--(3);
\draw[-] (x)--(y)--(x1);

\draw[-] (71)--(61)--(51)--(91)--(81)--(71);
\draw[-] (31)--(41)--(x1)--(01) -- (11)--(21)--(31);
\draw[-] (71)--(01);
\draw[-] (51)--(41);
\draw[-] (61)--(21);
\draw[-] (91)--(11);
\draw[-] (81)--(31);;
\end{tikzpicture}}}\qquad
\subfigure[An optimal identifying code of $G_{11}^k$.]{\label{fig:P11's-ID}\scalebox{0.8}{\begin{tikzpicture}[join=bevel,inner sep=0.5mm]
\node[graphnode](y) at (2.5,3) {};
\draw (y) node[above=0.15cm] {$y$};
\node[graphnode,fill](8) at (18:1) {};
\path (8)+(5,0) node[graphnode] (81) {};
\node[graphnode,fill](5) at (90:1) {};
\path (5)+(5,0) node[graphnode,fill] (51) {};
\node[graphnode](7) at (162:1) {};
\path (7)+(5,0) node[graphnode,fill] (71) {};
\node[graphnode](9) at (234:1) {};
\path (9)+(5,0) node[graphnode] (91) {};
\node[graphnode](6) at (306:1) {};
\path (6)+(5,0) node[graphnode] (61) {};
\node[graphnode](3) at (18:2) {};
\path (3)+(5,0) node[graphnode,fill] (31) {};
\node[graphnode](4) at (90:2) {};
\path (4)+(5,0) node[graphnode] (41) {};
\node[graphnode,fill](0) at (162:2) {};
\path (0)+(5,0) node[graphnode] (01) {};
\node[graphnode,fill](x) at (54:1.62) {};
\path (x)+(3.1,0) node[graphnode,fill] (x1) {};
\node[graphnode](1) at (234:2) {};
\path (1)+(5,0) node[graphnode,fill] (11) {};
\node[graphnode,fill](2) at (306:2) {};
\path (2)+(5,0) node[graphnode] (21) {};

\path (2.5,1.5) node  {$\ldots$};

\draw[-] (7)--(6)--(5)--(9)--(8)--(7);
\draw[-] (3)--(x)--(4)--(0) -- (1)--(2)--(3);
\draw[-] (7)--(0);
\draw[-] (5)--(4);
\draw[-] (6)--(2);
\draw[-] (9)--(1);
\draw[-] (8)--(3);
\draw[-] (x)--(y)--(x1);

\draw[-] (71)--(61)--(51)--(91)--(81)--(71);
\draw[-] (31)--(41)--(x1)--(01) -- (11)--(21)--(31);
\draw[-] (71)--(01);
\draw[-] (51)--(41);
\draw[-] (61)--(21);
\draw[-] (91)--(11);
\draw[-] (81)--(31);;
\end{tikzpicture}}}
\end{center}
\caption{\label{fig:P11's} The graph $G_{11}^k$.}
\end{figure}
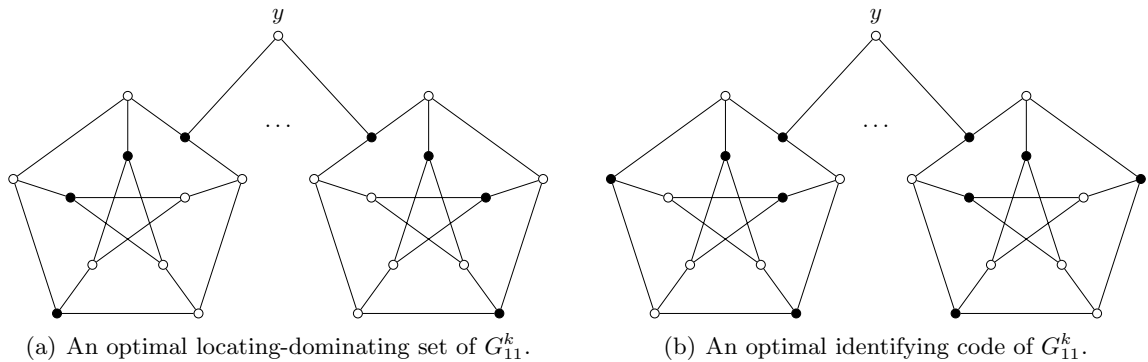

\subsection{Locating-dominating sets}

We now give constructions with large location-domination number. The first construction is based on copies of the $6$-cycle $C_6$.

\begin{proposition}\label{prop:example-LD-n/2}
There are infinitely many connected graphs $G$ of order $n$, girth~$5$ and minimum degree~$2$ with $\LD(G)=\frac{n-1}{2}$.
\end{proposition}
\begin{proof}
Consider the graph $G$ obtained from one vertex $x$ and $k\geq 2$ disjoint copies of $C_6$, each joined to $x$ by exactly one edge. We have $n=6k+1$, and we claim that $\LD(G)=3k$. It is easy to check that a set consisting of three vertices in each copy of $C_6$ (see Figure~\ref{fig:C_6's}) is locating-dominating. For the lower bound, assume that $D$ is an optimal locating-dominating set, and that $x\notin D$. Then, each copy of $C_6$ contains at least $\LD(C_6)=3$ vertices of $D$, and we are done. Hence, assume that $x\in D$. Each copy of $C_6$ has at least two vertices from $D$ (otherwise $D$ is not dominating). Assume some copy contains exactly two ($y,z$): then the neighbour of $x$ in that copy must be only dominated by $x$. Indeed, if this is not the case (say he is dominated also by $y$), there would be two vertices in this copy of $C_6$ that are not in $D$ but only dominated by $z$, a contradiction. But now observe that in the whole graph, at most one vertex of $V(G)\setminus D$ can be dominated only by $x$, hence all other copies of $C_6$ contain three vertices of $D$, and we are done.
\end{proof}

\begin{figure}[!ht]
\begin{center}
\scalebox{0.9}{\begin{tikzpicture}[join=bevel,inner sep=0.5mm]
\node[graphnode](x) at (1.5,2) {};
\draw (x) node[above=0.15cm] {$x$};
\node[graphnode](8) at (0:1) {};
\path (8)+(3,0) node[graphnode] (3) {};
\node[graphnode,fill](5) at (60:1) {};
\path (5)+(3,0) node[graphnode] (4) {};
\node[graphnode](7) at (120:1) {};
\path (7)+(3,0) node[graphnode,fill] (0) {};
\node[graphnode,fill](9) at (180:1) {};
\path (9)+(3,0) node[graphnode] (1) {};
\node[graphnode,fill](6) at (240:1) {};
\path (6)+(3,0) node[graphnode,fill] (2) {};
\node[graphnode](a) at (300:1) {};
\path (a)+(3,0) node[graphnode,fill] (b) {};

\path (1.5,0.5) node  {$\ldots$};

\draw[-] (8)--(5)--(7)--(9)--(6)--(a)--(8);
\draw[-] (3)--(4)--(0)--(1)--(2)--(b)--(3) (5)--(x)--(0);

\end{tikzpicture}}
\end{center}
\caption{\label{fig:C_6's} A family of connected graphs with location-domination number $\frac{n-1}{2}$.}
\end{figure}
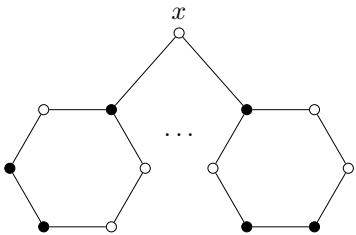

We will use the following lemma about the graph $P_{11}$.

\begin{lemma}\label{lemm:P11-LD}
Let $G$ be a graph of girth~5 containing a copy $P$ of $P_{11}$ as an induced subgraph, such that in $P$, only vertex $x$ has neighbours out of $P$. Let $D$ be a locating-dominating set of $G$. Then, we have $|D\cap V(P)|\geq 4$.
\end{lemma}
\begin{proof}
By contradiction, we assume that $D_P=D\cap V(P)$ has size~$3$. If $x\notin D_P$, then $D_P$ must form a locating-dominating set of $P\setminus\{x\}$. By Theorem~\ref{th:LB-Delta}, $\LD(G)\geq\frac{20}{6}>3$, a contradiction. Hence, $x\in D_P$. But now it is not possible to even dominate the remaining vertices with just two vertices, a contradiction.
\end{proof}

\begin{proposition}\label{prop:example-LD-conn-4/11}
There are infinitely many connected graphs $G$ of order $n$, girth~$5$ and minimum degree~$3$ with $\LD(G)=\frac{4}{11}(n-1)>0.363n$.
\end{proposition}
\begin{proof}
Consider the graph $G_{11}^k$ ($k\geq 3$) from Definition~\ref{def:G_11^k}, which has $n=11k+1$ vertices. A locating-dominating set of size~$4k$ is given by selecting vertices $\{x,3,6,9\}$ of each copy of $P_{11}$ (see Figure~\ref{fig:P11's-LD}). By Lemma~\ref{lemm:P11-LD}, this is optimal.
\end{proof}

The Heawood graph $H_{14}$ is a well-known Hamiltonian cubic vertex-transitive graph on $14$~vertices and with girth~$6$. Given its vertex set $\{0,1, \ldots, 13\}$, its edges are given by a Hamiltonian cycle $0-1-2-\ldots-13$ and $\{0,5\}$, $\{1,10\}$, $\{2,7\}$, $\{3,12\}$, $\{4,9\}$, $\{6,11\}$ and $\{8,13\}$. See Figure~\ref{fig:heawood} for an illustration.

\begin{figure}[!ht]
\begin{center}
\scalebox{0.9}{\begin{tikzpicture}[join=bevel,inner sep=0.5mm]
\node[graphnode](0) at (26:2) {};
\draw (0) node[right=0.1cm] {$0$};
\node[graphnode,fill](1) at (52:2) {};
\draw (1) node[above right=0.1cm] {$1$};
\node[graphnode](2) at (78:2) {};
\draw (2) node[above=0.1cm] {$2$};
\node[graphnode](3) at (104:2) {};
\draw (3) node[above=0.1cm] {$3$};
\node[graphnode,fill](4) at (129:2) {};
\draw (4) node[above left=0.1cm] {$4$};
\node[graphnode](5) at (155:2) {};
\draw (5) node[left=0.1cm] {$5$};
\node[graphnode,fill](6) at (181:2) {};
\draw (6) node[left=0.1cm] {$6$};
\node[graphnode](7) at (206:2) {};
\draw (7) node[left=0.1cm] {$7$};
\node[graphnode,fill](8) at (232:2) {};
\draw (8) node[below left=0.1cm] {$8$};
\node[graphnode](9) at (258:2) {};
\draw (9) node[below=0.1cm] {$9$};
\node[graphnode,fill](10) at (283:2) {};
\draw (10) node[below=0.1cm] {$10$};
\node[graphnode](11) at (309:2) {};
\draw (11) node[below right=0.1cm] {$11$};
\node[graphnode](12) at (335:2) {};
\draw (12) node[below right=0.1cm] {$12$};
\node[graphnode,fill](13) at (361:2) {};
\draw (13) node[right=0.1cm] {$13$};

\draw[-] (0)--(1)--(2)--(3)--(4)--(5)--(6)--(7)--(8)--(9)--(10)--(11)--(12)--(13)--(0);
\draw[-] (0)--(5) (1)--(10) (2)--(7) (3)--(12) (4)--(9) (6)--(11) (8)--(13);
\end{tikzpicture}}
\end{center}
\caption{\label{fig:heawood} The Heawood graph with a minimum locating-dominating set (black vertices).}
\end{figure}
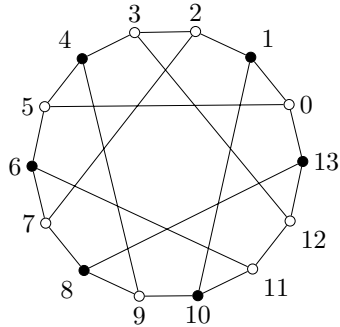

\begin{proposition}\label{prop:heawood-LD}
The Heawood graph $H_{14}$ has $\LD(H_{14})=6=\frac{3}{7}n>0.428n$.
\end{proposition}
\begin{proof}
A locating-dominating set of size~$6$ is for example $\{1,4,6,8,10,13\}$. 

We now prove that no locating-dominating set of size~$5$ exists. Assume by contradiction that there is a locating-dominating set $D$ of $H_{14}$ of size~$5$. Let $m(D)$ and $m(D,S)$ count the number of edges between vertices of $D$ and the edges between $D$ and $S=V(H_{14})\setminus D$, respectively. Since at most $|D|$ vertices from $S$ can be dominated by a single vertex of $D$, we have $m(D,S)\geq |D|+2(|S|-|D|)=13$. On the other hand, since $H_{14}$ is cubic, $m(D,S)=15-2m(D)$. Hence, we have $m(D)\leq 1$.

Therefore, we have at least three vertices in $D$ that are adjacent only to vertices of $S$. Since $H_{14}$ is vertex-transitive, we assume without loss of generality that vertex $0$ is such a vertex. Among the neighbours of $0$ (vertices $1,5,13$), at most one is dominated only by $0$.

Assume that one of them is in that case. By the symmetries of the graph, there are automorphisms pairwise exchanging edges $\{0,1\},\{0,5\},\{0,13\}$. Hence, without loss of generality, we can assume that vertex $5$ is $1$-dominated, but vertices $1,13$ are $2$-dominated. Hence, vertices $4,6\notin D$ but at least one vertex among $2,10$ and $8,12$ belongs to $D$, respectively. Moreover, in order to dominate vertices $4$ and $6$, one of $3,9$ and $7,11$ belongs to $D$, respectively. Since these four sets are disjoint and $|D|=5$, $D$ contains \emph{exactly} one of each.

We first assume that $2\in D$: hence
$10\notin D$. If also $7\in D$ (and
$11\notin D$), both $9,12\in D$ in order to dominate
10 and 11, respectively. Then $D=\{0,2,7,9,12\}$ but $4,10$
are both dominated only by 9, a contradiction. Hence,
$7\notin D$ and $11\in D$. Then, $9\in D$
in order to separate $6,10$; then, $3\notin D$ and
$12\in D$, otherwise $6,12$ are not separated. Hence
$D=\{0,2,9,11,12\}$ but $4,8$ are both dominated only by 9,
a contradiction.

Hence, $2\notin D$ and $10\in D$. If $3\notin D$, then $9\in D$ and moreover $12\in D$ in order to dominate $3$ (hence $8\notin D$). Since $7$ is dominated, $7$ is the last vertex of $D$. But then $2,6$ are both dominated only by $7$, a contradiction. Hence, $3\in D$ and $9\notin D$. To separate $2,4$, $7\in D$ (hence $11\notin D$). Then $8$ is the last vertex of $D$, otherwise it would not be separated by $6$. But then $4,12$ are not separated, a contradiction.

Therefore, we can assume that all neighbours of $0$ are $2$-dominated. Hence, at least one vertex among $\{2,10\}$, $\{4,6\}$ and $\{8,12\}$, respectively, belongs to $D$. Assume first that $2\in D$. Then, in order for $10$ to be dominated, one of $9,10,11$ belongs to $D$. Then, exactly one of $8,12$ belongs to $D$. If $10\in D$, one of $8,12$ would not be dominated, a contradiction. If $9\in D$, then $12\in D$ (otherwise it is not dominated). But then, both $8,10$ are only dominated by $9$, a contradiction. A similar contradiction follows if $11\in D$.

Hence, $2\notin D$, and $10\in D$. Then, (exactly) one of $3,7$ belongs to $D$, otherwise $2$ is not dominated. If $3\in D$ and $7\notin D$, $8\in D$ (otherwise $8$ is not dominated). Since $6$ must be dominated, $6$ itself is the last vertex of $D$; but then, $4,12$ are both only dominated by $3$, a contradiction. Hence, If $7\in D$ and $3\notin D$. Then, $12\in D$ (otherwise it is not dominated). Hence, $8\notin D$. But now, both $2,8$ are only dominated by $7$, a contradiction.

Therefore, $D$ does not exist, which completes the proof.
\end{proof}

\subsection{Identifying codes}

We now give constructions with large identifying code number. We start with a construction based on the 5-cycle $C_5$, which has identifying code number~$3$~\cite{BCHL04}.

\begin{proposition}\label{prop:example-ID-3n/5}
There are infinitely many connected graphs $G$ of order $n$, girth~$5$ and minimum degree~$2$ with $\M(G)=\frac{3}{5}(n-1)$.
\end{proposition}
\begin{proof}
Consider a vertex $x$ attached to $k\geq 2$ copies of $C_5$ via one of each copy's vertex (Figure~\ref{fig:C_5's}). The set formed by three consecutive vertices of each copy of $C_5$ (centered in the neighbour of $x$) is clearly an identifying code. For the lower bound, assume that some copy contains at most two vertices of an identifying code $C$. Then they must be non-adjacent (otherwise some vertex is not dominated). But then at least one of these two vertices is not separated from one of its neighbours, a contradiction. Hence each copy of $C_5$ contains at least three vertices of $C$, proving the bound.
\end{proof}

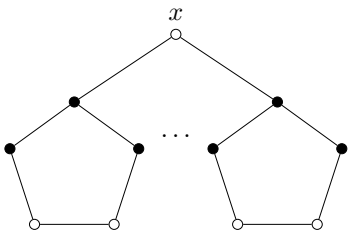
\begin{figure}[!ht]
\begin{center}
\scalebox{0.9}{\begin{tikzpicture}[join=bevel,inner sep=0.5mm]
\node[graphnode](x) at (1.5,2) {};
\draw (x) node[above=0.15cm] {$x$};
\node[graphnode,fill](8) at (18:1) {};
\path (8)+(3,0) node[graphnode,fill] (3) {};
\node[graphnode,fill](5) at (90:1) {};
\path (5)+(3,0) node[graphnode,fill] (4) {};
\node[graphnode,fill](7) at (162:1) {};
\path (7)+(3,0) node[graphnode,fill] (0) {};
\node[graphnode](9) at (234:1) {};
\path (9)+(3,0) node[graphnode] (1) {};
\node[graphnode](6) at (306:1) {};
\path (6)+(3,0) node[graphnode] (2) {};

\path (1.5,0.5) node  {$\ldots$};

\draw[-] (8)--(5)--(7)--(9)--(6)--(8);
\draw[-] (3)--(4)--(0)--(1)--(2)--(3) (5)--(x)--(4);

\end{tikzpicture}}
\end{center}
\caption{\label{fig:C_5's} A family of connected graphs with identifying code number $\frac{3}{5}(n-1)$.}
\end{figure}

The following lemma is about the graph $P_{11}$.

\begin{lemma}\label{lemm:P11}
Let $G$ be an identifiable graph of girth~$5$ containing a copy $P$ of $P_{11}$ as an induced subgraph, such that in $P$, only vertex $x$ has neighbours out of $P$. Let $C$ be an identifying code of $G$ and $C\cap V(P)=C_P$. Then:\begin{itemize}
\item[(i)] $|C_P|\geq 4$;
\item[(ii)] if $|C_P|=4$, then $x$ is \emph{only} dominated by a vertex $y\notin V(P)$;
\end{itemize}
\end{lemma}
\begin{proof}
(i) By contradiction, assume that $|C_P|=3$. If $C_P$ induces a connected graph, then one can check that there are some non-dominated vertices in $P$. Hence, by Lemma~\ref{lemma:girth5-ID}(i), either $C_P$ induces a $K_2$ containing $x$ and an isolated vertex, or three isolated vertices. In both cases some vertices of $P$ would not be separated, a contradiction.

(ii) Assume that $|C_P|=4$ and by contradiction, that $x$ is dominated by a vertex of $C_P$. If $x\notin C_P$, then $C_P$ must form an identifying code of $P\setminus\{x\}$. Then, the bound $\M(G)\geq\frac{2n}{\Delta+2}$ of Theorem~\ref{th:LB-Delta} is tight, and by the same theorem, all vertices in $C_P$ have degree~$3$ in $P\setminus\{x\}$. Hence the neighbours of $x$ do not belong to $C_P$, a contradiction. Hence, $x\in C_P$.

Let $m(C_P)$ and $m(C_P,S)$ count the number of edges between vertices of $C_P$ and edges between vertices of $C_P$ and $S=V(P)\setminus C_P$, respectively. Let $i$ denote the number of vertices in $C_P$ that are not adjacent to any other vertex of $C_P$ (note that $0\leq i\leq 2$ since $x\in C_P$ and $x$ is dominated by a vertex of $C_P$). Then, we have $m(C_P)=4-i-1$ (indeed $C_P$ must induce a forest). We also have $m(C_P,S)=11-2m(C_P)$ (since $x\in C_P$ and has degree~$2$ in $P$). We get that $m(C_P,S)=5+2i$. On the other hand, at most $4-i$ vertices in $S$ can be $1$-dominated, and the other ones must be at least $2$-dominated. Since $|S|=7$, we get $m(C_P,S)\geq 4-i + 2(7-(4-i))=10+i$. Putting both inequalities together, we get that $i\geq 5$, a contradiction.
\end{proof}

\begin{proposition}\label{prop:example-ID-5n/11}
There are infinitely many connected graphs $G$ of order $n$, girth~$5$ and minimum degree~$3$ with $\M(G)=\frac{5}{11}(n-1)>0.454n$.
\end{proposition}
\begin{proof}
Consider the graph $G_{11}^k$ from Definition~\ref{def:G_11^k}. An identifying code of size $5k$, formed by vertices $\{x,2,4,7,9\}$ of each copy of $P_{11}$, is illustrated in Figure~\ref{fig:P11's-ID}. Now, consider an identifying code $C$ of the graph. By Lemma~\ref{lemm:P11}(i), every copy of $P_{11}$ contains at least four vertices of $C$. By Lemma~\ref{lemm:P11}(ii), for each copy of $P_{11}$ containing \emph{exactly} four code-vertices, then $y\in C$ and vertex $x$ is dominated only by $y$. Hence there can be only one such copy, proving the lower bound.
\end{proof}

We now define a cubic graph on $12$~vertices with girth~$5$.

\begin{definition}
Let $G_{12}$ be the $12$-vertex graph with vertex set $\{0,1, \ldots, 11\}$ and edges given by a hamilitonian cycle $0-1-2-\ldots-11-0$ and $\{0,4\}$, $\{1,8\}$, $\{2,6\}$, $\{3,10\}$, $\{5,9\}$, and $\{7,11\}$.
\end{definition}

An illustration is given in Figure~\ref{fig:g12}. We remark that, alternatively, $G_{12}$ can be obtained from the Petersen graph by subdividing two edges that are at maximum distance (i.e. distance~$2$) from each other and joining the two new vertices by an edge. A third way is to take the Heawood graph, delete two adjacent vertices $x,y$ and adding an edge between the two neighbours of $x$ and an edge between the two neighbours of $y$.

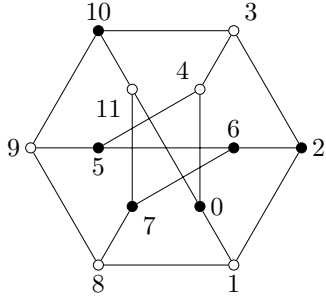
\begin{figure}[!ht]
\begin{center}
\scalebox{0.9}{\begin{tikzpicture}[join=bevel,inner sep=0.5mm]
\node[graphnode,fill](2) at (0:2) {};
\draw (2) node[right=0.1cm] {$2$};
\node[graphnode](3) at (60:2) {};
\draw (3) node[above right=0.1cm] {$3$};
\node[graphnode,fill](10) at (120:2) {};
\draw (10) node[above=0.1cm] {$10$};
\node[graphnode](9) at (180:2) {};
\draw (9) node[left=0.1cm] {$9$};
\node[graphnode](8) at (240:2) {};
\draw (8) node[below=0.1cm] {$8$};
\node[graphnode](1) at (300:2) {};
\draw (1) node[below=0.1cm] {$1$};

\node[graphnode,fill](6) at (0:1) {};
\draw (6) node[above=0.1cm] {$6$};
\node[graphnode](4) at (60:1) {};
\draw (4) node[above left=0.1cm] {$4$};
\node[graphnode](11) at (120:1) {};
\draw (11) node[below left=0.1cm] {$11$};
\node[graphnode,fill](5) at (180:1) {};
\draw (5) node[below=0.1cm] {$5$};
\node[graphnode,fill](7) at (240:1) {};
\draw (7) node[below right=0.1cm] {$7$};
\node[graphnode,fill](0) at (300:1) {};
\draw (0) node[right=0.1cm] {$0$};

\draw[-] (8)--(1)--(2)--(3)--(10)--(9)--(8)
         (6)--(5)--(4)--(0)--(11)--(7)--(6)
         (2)--(6) (3)--(4) (10)--(11) (5)--(9) (7)--(8) (0)--(1);
\end{tikzpicture}}
\end{center}
\caption{\label{fig:g12} The graph $G_{12}$ with a minimum identifying code (black vertices).}
\end{figure}

\begin{proposition}\label{prop:G12-ID}
The graph $G_{12}$ has $\M(G_{12}) = 6 = \frac{n}{2}$.
\end{proposition}
\begin{proof}
An identifying code of size $6$ is given for instance by the set $\{0,2,5,6,7,10\}$, implying that $\M(G_{12}) \le 6$.

To prove our claim, it is sufficient to show that there is no identifying code on $5$ vertices. Assume for contradiction that there is an identifying code $C$ of $G_{12}$ of size $5$. Let $I$ be the set of isolated vertices in $C$, $S = V(G_{12})\setminus C$. Let $m(C)$ and $m(C,S)$ count the number of edges between vertices of $C$ and the edges between $C$ and $S$, respectively. By Lemma~\ref{lemma:girth5-ID}, at least $|S|-|C\setminus I|$ vertices from $S$ have to be $2$-dominated. Hence, there are at least $|C\setminus I|+2(|S|-|C\setminus I|)$ edges from $S$ to $C$. On the other hand, there are $3|C|-2m(C) = 15 - 2m(C)$ edges from $C$ to $S$. Hence,
\begin{align*}
15-2m(C) = m(C,S) &\ge |C\setminus I|+ 2(|S|-|C\setminus I|) \\
                                    &= 2|S| - |C\setminus I|\\
                                    &= 2(12-|C|)-|C|+|I| =9+|I|,
\end{align*}
which gives 
\begin{align}
m(C) \le 3 - \frac{|I|}{2}.\label{eq:m(C)}
\end{align}

By Lemma~\ref{lemma:girth5-ID}(i), the subgraph induced by $C$ consists either of a single component of order~$5$, a component of order~$4$ and an isolated vertex, a component of order~$3$ and two isolated vertices, or $C$ is an independent set. We distinguish now between these cases.

\vspace{0.2cm}\noindent{\it Case a: $C$ consists of a single component of order $5$.} Then $m(C)\geq 4$ and thus, by Inequality~\eqref{eq:m(C)}, $4\leq 3$, which is a contradiction.

\vspace{0.2cm}\noindent{\it Case b: $C$ consists of a component of order $4$ and an isolated vertex.} Again, by Inequality~\eqref{eq:m(C)}, $3\leq m(C)\leq 2.5$, a contradiction.

\vspace{0.2cm}\noindent {\it Case c: $C$ consists of a component $C_c$ of order~$3$ and two isolated vertices $x$ and $y$.}
Then $C_c$ is a path of length~$2$, say $uvw$, and $m(C)=2=3-\frac{|I|}{2}$, giving equality in the above inequality chain. Hence, three vertices from $S$ have exactly one neighbour in $C$, while the other four have exactly two neighbours in $C$. With $S = \{s_1,s_2,s_3,s_4,s_5,s_6,s_7\}$, let us say that $s_1$, $s_2$, and $s_3$ are the vertices being dominated once and let $\{s_1,u\}$, $\{s_2,v\}$, $\{s_3,w\}$ be the edges hereby involved. Then the edges incident with $v$ have been all assigned, while $u$ and $w$ can still contribute dominating one more vertex from $S$. However, to $2$-dominate the vertices in $\{s_4,s_5,s_6,s_7\}$, necessarily two of them will be adjacent to both $x$ and $y$, building a cycle of length~$4$, which is not allowed. Thus, this case is not possible.

\vspace{0.2cm}\noindent{\it Case d: $C = I = \{x_1,x_2,x_3,x_4,x_5\}$.} By Lemma~\ref{lemma:girth5-ID}(ii), all vertices from $S$ have to be $2$-dominated by $I$, and hence $m(C,S)\geq 14$. But since $G_{12}$ is cubic, each vertex of $S$ is incident to at most one further edge. Since $|S|=7$, at most six vertices in $S$ can be paired, and $m(C,S)\geq 15$. On the other hand, each vertex of $C$ has three neighbours, hence $m(C,S)\leq 15$. This implies that while one vertex from $S$, say $s_7$, has exactly three neighbours in $I$, the other six vertices from $S$ have exactly two neighbours in $I$. Since $G_{12}$ is cubic, the vertices in $\{s_1,s_2,s_3,s_4,s_5,s_6\}$ are paired by a matching, say $\{s_1,s_2\}$, $\{s_3,s_4\}$, $\{s_5,s_6\}$. Consider the edge $\{s_1,s_2\}$ and its neighbours in $I$, say $\{x_1,x_2,x_3,x_4\}$. Going through all edges from the graph $G_{12}$ and considering their four independent neighbours, there are only two possibilities where the corresponding independent sets can be completed to an independent set of size~$5$. These are the edges $\{3,4\}$ and $\{7,8\}$ which give each two possible independent sets $\{0,2,5,7,10\}$, $\{0,2,5,8,10\}$ and $\{1,3,6,9,11\}$, $\{1,4,6,9,11\}$. Hence, $C=I$ has to be one of these sets. However, it is easy to check that none of them is an identifying code.

\vspace{0.2cm}
 Hence, $G_{12}$ has no identifying code of size $5$ and $\M(G_{12}) = 6$.
\end{proof}

\section{Conclusion}\label{sec:conclu}

We proved the two tight upper bounds $\LD(G)\leq\frac{n}{2}$ and $\M(G) \leq \frac{5}{7}n$ for graphs $G$ of girth at least~$5$ and minimum degree at least~$2$, as well as improved bounds for cubic graphs. While the first bound is asymptotically tight for large values of $n$, we do not know whether this holds for the latter one. 

For minimum degree at least~$3$, either our bounds are not tight, or we have not found the graphs with highest value of parameters $\LD$ and $\M$. In particular, the question whether for every graph $G$ of girth at least~$5$ and minimum degree at least~$3$, we have $\M(G)\leq\frac{n}{2}$ seems intriguing. By Proposition~\ref{prop:G12-ID} this would be tight for the graph $G_{12}$. Though we have tried to get better bounds when $\delta\geq 3$, it seems that our technique is not powerful enough for such an improvement (at least without any new idea).

To conclude, we remark that another interesting question would be to conduct a similar study for the \emph{open location-domination number} and the \emph{locating-total domination number}, related concepts introduced in~\cite{SS10} and~\cite{hhh06}, respectively.

\vspace{0.2cm}
\noindent\textbf{Acknowledgements.} We thank the anonymous referee for carefully reading the paper, therefore helping to improve its presentation.

\end{document}